\documentclass[11pt]{article}
\usepackage{amsmath,amssymb}

\def\Z{\mathbb{Z}}
\def\R{\mathbb{R}}

\def\C{C^{\infty}(M)}

\newtheorem{definition}{Definition}[section]
\newtheorem{lemma}[definition]{Lemma}
\newtheorem{proposition}[definition]{Proposition}
\newtheorem{theorem}[definition]{Theorem}
\newtheorem{remark}[definition]{Remark}

\newtheorem{example}[definition]{Example}
\newenvironment{proof}{\noindent{\bf Proof.}}{\hfill $\blacklozenge$}
\newtheorem{th-lep}[definition]{Lepage's decomposition theorem}

\begin{document}

\title{Poisson brackets with prescribed Casimirs}

\author{Pantelis A. Damianou and Fani Petalidou}
\date{}
\maketitle

\vskip 1 cm

\begin{center}
\emph{Dedicated to Giuseppe Marmo, on the occasion of his  65$^{th}$ birthday.}
\end{center}

\begin{abstract}
We consider the problem of constructing Poisson brackets on smooth manifolds $M$ with prescribed Casimir functions. If $M$ is of even dimension, we achieve our construction by considering a suitable almost symplectic structure  on $M$, while, in the case where $M$ is of odd dimension, our objective  is achieved by using a convenient  almost cosymplectic structure. Several examples and applications are presented.
\end{abstract}

\vspace{5mm} \noindent {\bf{Keywords: }}{Poisson bracket, Casimir function, almost symplectic structure, almost cosymplectic structure.}

\vspace{3mm} \noindent {\bf{MSC (2010):}} 53D17, 53D15.

\section{Introduction}
A \emph{Poisson bracket} on the space $\C$ of smooth functions on a smooth manifold $M$ is a skew-symmetric, bilinear map,
\begin{equation*}
\{\cdot,\cdot\} : \C \times \C \to \C,
\end{equation*}
that verifies the Jacobi identity and is a biderivation. Thus, $(\C, \{\cdot,\cdot\})$ has the structure of a Lie algebra. This notion has been introduced in the framework of classical mechanics by S. D. Poisson, who discovered the natural symplectic bracket on $\R^{2n}$ \cite{poi}, a notion that  was later generalized to manifolds of arbitrary dimension by S. Lie \cite{lie}. The interest in this subject, motivated by the important role of Poisson structures in Hamiltonian dynamics, was revived during the last 35 years, after the publication of the fundamental works of A. Lichn\'erowicz \cite{damianou:lch1}, A. Kirillov \cite{kir} and A. Weinstein \cite{wei}, and Poisson geometry has emerged as a major branch of modern differential geometry. The pair $(M, \{\cdot,\cdot\})$ is called a \emph{Poisson manifold} and is foliated by symplectic immersed submanifolds, the \emph{symplectic leaves}. The functions in  the center of $(\C, \{\cdot,\cdot\})$, i.e., the elements $f\in \C$ such that $\{f,\cdot\}=0$, are called the \emph{Casimirs} of the Poisson bracket $\{\cdot,\cdot\}$ and they form the space of first integrals of the symplectic leaves. For this reason, Casimir invariants have acquired a dominant role in the study of integrable systems defined on a  manifold $M$ and in the theory of the local structure of Poisson manifolds \cite{wei}.

To introduce the problem we remark that, for an arbitrary smooth function $f$ on $\R^3$, the bracket
\begin{equation} \label{damianou:equation1}
\{x, y\}={\partial  f \over \partial  z}, \qquad
\{x, z \}=-{\partial  f \over \partial  y} \qquad \mathrm{and} \qquad
\{y, z\}= {\partial  f \over \partial  x}
\end{equation}
is Poisson and it admits $f$ as Casimir. Clearly, if $\Omega = dx\wedge dy\wedge dz$ is the standard volume element on $\R^3$, then the bracket (\ref{damianou:equation1}) can be written as
\begin{equation*}
\{x,y\}\Omega = dx\wedge dy \wedge df, \quad \{x,z\}\Omega = dx\wedge dz \wedge df, \quad \{y,z\}\Omega = dy\wedge dz \wedge df.
\end{equation*}
More generally, let $f_1,f_2,\ldots, f_l$ be functionally independent smooth functions on $\mathbb{R}^{l+2}$ and $\Omega$ a non-vanishing $(l+2)$-smooth form on $\mathbb{R}^{l+2}$. Then, the formula
\begin{equation} \label{damianou:equation2}
\{ g, h \} \Omega =fdg \wedge dh \wedge df_1 \wedge \ldots \wedge df_l, \quad \quad g,\,h \in C^\infty(\R^{l+2}),
\end{equation}
defines a Poisson bracket on $\mathbb{R}^{l+2}$ with $f_1, \ldots, f_l $ as Casimir invariants. In addition, the symplectic leaves of (\ref{damianou:equation2}) have dimension at most $2$. The Jacobian Poisson structure (\ref{damianou:equation2}) (the bracket $\{g,h\}$ is equal, up to a coefficient function $f$, with the usual Jacobian determinant of $(g,h,f_1,\ldots,f_l)$) appeared in \cite{damianou:Dam89} in  1989  where it was attributed to H. Flaschka and T. Ratiu. The first explicit proof of this result was given in \cite{damianou:Grab93}, while the first application of formula (\ref{damianou:equation2}) is presented in \cite{damianou:Dam89, damianou:Dam96} in conjunction with transverse Poisson structures to subregular nilpotent orbits of $\mathfrak{gl}(n,\mathbb{C})$, $n\leq 7$. It was shown that these transverse Poisson structures which are usually computed using Dirac's constraint formula can be calculated much more easily using the Jacobian Poisson structure  (\ref{damianou:equation2}).  This fact was extended to any semisimple Lie algebra in \cite{damianou:Dam07}. In the same paper it is also proved that, after a suitable change of coordinates, the above referred transverse Poisson structures is reduced to  a 3-dimensional structure of type (\ref{damianou:equation1}). We believe that for the other type of orbits, e.g. the minimal orbit and all the other intermediate orbits, one can compute the transverse Poisson structures using the results of the present paper. However, this study will be the subject of a future work. Another interesting application of formula (\ref{damianou:equation2}) appears  in \cite{or}, where the polynomial Poisson algebras with some regularity conditions are studied. We also mention the   study  of a family of rank 2 Poisson structures  in \cite{bermejo}.

The purpose of this paper is to extend the  formula of type (\ref{damianou:equation2}) in the more general case of  higher rank Poisson brackets. The  problem can be formulated as follows: \emph{Given $(m-2k)$ smooth functions $f_1,\ldots,f_{m-2k}$ on an  $m$-dimensional smooth manifold $M$, functionally independent almost everywhere, describe the Poisson brackets $\{\cdot,\cdot\}$ on $\C$ of rank at most $2k$ which have $f_1,\ldots,f_{m-2k}$ as Casimirs.} Firstly, we investigate this problem in the case where $m=2n$, i.e., $M$ is of even dimension. We assume that $M$ is endowed with a suitable  almost symplectic structure $\omega_0$ and we prove that (Theorem \ref{damianou:THEOREM}) a Poisson bracket $\{\cdot,\cdot\}$ on $\C$ with the required  properties is defined, for any $h_1,h_2 \in C^{\infty}(M)$, by the formula
\begin{equation*}
\{h_1,h_2\} \Omega =  -\frac{1}{f} dh_1 \wedge dh_2 \wedge (\sigma + \frac{g}{k-1}\omega_0) \wedge \frac{\omega_0^{k-2}}{(k-2)!}\wedge df_1\wedge\ldots \wedge df_{2n-2k},
\end{equation*}
where $\Omega=\displaystyle{\frac{\omega_0^n}{n!}}$ is a volume element on $M$, $f$ satisfies $f^2 = \det \big(\{f_i,f_j\}_{_0}\big)\neq 0$ ($\{\cdot,\cdot\}_{_0}$ being the bracket defined by $\omega_0$ on $\C$), $\sigma$ is a $2$-form on $M$ satisfying certain  special requirements (see, Proposition \ref{damianou:theorem-cond-delta-sigma}) and $g = i_{\Lambda_0}\sigma$ \footnote{$\Lambda_0$ being the bivector field on $M$ associated to $\omega_0$.}. We proceed  by considering the case where $M$ is an odd-dimensional manifold, i.e., $m=2n+1$, and we establish a similar formula for the Poisson brackets on $\C$ with the prescribed properties. For this construction, we assume that $M$ is equipped with a suitable almost cosymplectic structure $(\vartheta_0,\Theta_0)$ and with the volume form $\Omega = \vartheta_0 \wedge \displaystyle{\frac{\Theta_0^n}{n!}}$. Then, we show that (Theorem \ref{THEOREM-ODD}) a Poisson bracket $\{\cdot,\cdot\}$ on $\C$ with $f_1,\ldots,f_{2n+1-2k}$ as Casimirs functions is defined, for any $h_1,h_2 \in C^{\infty}(M)$, by the formula
\begin{equation*}
\{h_1,h_2\} \Omega =  -\frac{1}{f} dh_1 \wedge dh_2 \wedge (\sigma + \frac{g}{k-1}\Theta_0) \wedge \frac{\Theta_0^{k-2}}{(k-2)!}\wedge df_1\wedge\ldots \wedge df_{2n+1-2k},
\end{equation*}
where $f$ is given by (\ref{f-odd}), $\sigma$ is a $2$-form on $M$ satisfying certain  particular conditions (see, Proposition \ref{prop-sigma-odd}), and $g=i_{\Lambda_0}\sigma$ \footnote{$\Lambda_0$ being the bivector field on $M$ associated to $(\vartheta_0,\Theta_0)$.}.

The proofs of the main results are given in section 3. Section 2 consists of preliminaries and fixing the notation, while in section 4 we present several applications of our formul{\ae} on Dirac brackets, on almost Poisson brackets associated to nonholonomic systems and on Toda and Volterra lattices.

\section{Preliminaries}
We start by fixing our notation and by recalling the most important notions and formul{\ae} needed in this paper. Let $M$ be a real smooth $m$-dimensional manifold, $TM$ and $T^\ast M$ its tangent and cotangent bundles and $C^{\infty}(M)$ the space of smooth functions on $M$. For each $p\in \mathbb{Z}$, we denote by $\mathcal{V}^p(M)$ and $\Omega^p(M)$ the spaces of smooth sections, respectively, of $\bigwedge^p TM$ and $\bigwedge^p T^\ast M$. By convention, we set $\mathcal{V}^p(M) = \Omega^p(M) = \{0\}$, for $p<0$, $\mathcal{V}^0(M) = \Omega^0(M) = C^{\infty}(M)$, and, taking into account the skew-symmetry, we have $\mathcal{V}^p(M) = \Omega^p(M) = \{0\}$, for $p>m$. Finally, we set $\mathcal{V}(M)=\oplus_{p\in \mathbb{Z}}\mathcal{V}^p(M)$ and $\Omega(M) = \oplus_{p\in \mathbb{Z}}\Omega^p(M)$.

\subsection{From multivector fields to differential forms and back}
There is a natural \emph{pairing} between the elements of $\Omega(M)$ and $\mathcal{V}(M)$, i.e.,  the $C^{\infty}(M)$-bilinear map $\langle \cdot, \cdot \rangle : \Omega(M) \times \mathcal{V}(M)  \to  C^{\infty}(M)$, $(\eta, P)  \mapsto  \langle \eta,P\rangle$, defined as follows: For any $\eta \in \Omega^q(M)$ and $P\in \mathcal{V}^p(M)$ with $p \neq q$, $\langle \eta, P\rangle =0$; for any $f,g\in \Omega^0(M)$, $\langle f, g\rangle =fg$; while, if $\eta = \eta_1\wedge \eta_2\wedge \ldots \wedge \eta_p \in \Omega^p(M)$ is a decomposable $p$-form ($\eta_i\in \Omega^1(M)$) and $P = X_1\wedge X_2\wedge \ldots \wedge X_p$ is a decomposable $p$-vector field ($X_i\in \mathcal{V}^1(M)$),
\begin{equation*}
\langle \eta, P\rangle = \langle \eta_1\wedge \eta_2\wedge \ldots \wedge\eta_p, X_1\wedge X_2\wedge \ldots \wedge X_p \rangle = \det(\langle \eta_i,X_j\rangle).
\end{equation*}
The above definition is extended to the nondecomposable forms and multivector fields by bilinearity in a unique way. Precisely, for any $\eta\in \Omega^p(M)$ and $X_1,\ldots,X_p \in \mathcal{V}^1(M)$,
\begin{equation*}
\langle \eta, X_1\wedge X_2\wedge \ldots \wedge X_p\rangle = \eta (X_1,X_2, \ldots, X_p).
\end{equation*}
Similarly, for $P\in \mathcal{V}^p(M)$ and $\eta_1,\eta_2,\ldots,\eta_p \in \Omega^1(M)$,
\begin{equation*}
\langle \eta_1\wedge \eta_2 \wedge \ldots \wedge \eta_p, P \rangle = P(\eta_1,\eta_2,\ldots,\eta_p).
\end{equation*}

We adopt the following convention for the \emph{interior product $i_P: \Omega(M) \to \Omega(M)$ of differential forms by a $p$-vector field} $P$, viewed as a $C^{\infty}(M)$-linear endomorphism of $\Omega(M)$ of degree $-p$. If $P=X \in \mathcal{V}^1(P)$ and $\eta$ is a $q$-form, $i_X\eta$ is the element of $\Omega^{q-1}(M)$ defined, for any $X_1,\ldots,X_{q-1}\in \mathcal{V}^1(M)$, by
\begin{equation*}
(i_X\eta)(X_1,\ldots,X_{q-1}) = \eta (X,X_1,\ldots,X_{q-1}).
\end{equation*}
If $P=X_1\wedge X_2\wedge \ldots \wedge X_p$ is a decomposable $p$-vector field, we set
\begin{equation*}
i_P\eta = i_{X_1\wedge X_2\wedge \ldots \wedge X_p}\eta = i_{X_1}i_{X_2}\ldots i_{X_p}\eta.
\end{equation*}
More generally, recalling that each $P\in \mathcal{V}^p(M)$ can be locally written as the sum of decomposable $p$-vector fields, we define as $i_P\eta$, with $\eta \in \Omega^{q}(M)$ and $q\geq p$, to be the unique element of $\Omega^{q-p}(M)$ such that, for any $Q\in \mathcal{V}^{q-p}(M)$,
\begin{equation}\label{damianou:sign-inter. prod}
\langle i_P\eta, Q\rangle = (-1)^{(p-1)p/2}\langle \eta, P\wedge Q\rangle.
\end{equation}

Similarly, we define the \emph{interior product $j_{\eta} : \mathcal{V}(M) \to \mathcal{V}(M)$ of multivector fields by a $q$-form $\eta$}. If $\eta = \alpha \in \Omega^1(M)$ and $P\in \mathcal{V}^p(M)$, then $j_{\alpha}P$ is the unique $(p-1)$-vector field on $M$ given, for any $\alpha_1,\ldots,\alpha_{p-1}$, by
\begin{equation*}
(j_{\alpha}P)(\alpha_1,\ldots,\alpha_{p-1})= P(\alpha_1,\ldots,\alpha_{p-1}, \alpha).
\end{equation*}
Moreover, if $\eta = \alpha_1\wedge \alpha_2 \wedge \ldots \wedge \alpha_q$ is a decomposable $q$-form, we set
\begin{equation*}
j_{\eta}P = j_{\alpha_1\wedge \alpha_2 \wedge \ldots \wedge \alpha_q}P = j_{\alpha_1}j_{\alpha_2}\ldots j_{\alpha_q}P.
\end{equation*}
Hence, using the fact that any $\eta \in \Omega^q(M)$ can be locally written as the sum of decomposable $q$-forms, we define $j_{\eta}$ to be the $C^{\infty}(M)$-linear endomorphism of $\mathcal{V}(M)$ of degree $-q$ which associates, with each $P\in \mathcal{V}^p(M)$ ($p\geq q$), the unique $(p-q)$-vector field $j_{\eta}P$ defined, for any $\zeta\in \Omega^{p-q}(M)$, by
\begin{equation*}
\langle \zeta, j_{\eta}P\rangle = \langle \zeta\wedge \eta, P\rangle.
\end{equation*}
If the degrees of $\eta$ and $P$ are equal, i.e., $q=p$, the interior products $j_{\eta}P$ and $i_P\eta$ are, up to sign, equal:
\begin{equation*}
j_{\eta}P = (-1)^{(p-1)p/2}i_P\eta = \langle \eta, P\rangle.
\end{equation*}

\vspace{2mm}

The \emph{Schouten bracket} $[\cdot,\cdot] : \mathcal{V}(M)\times \mathcal{V}(M) \to \mathcal{V}(M)$, which is a natural extension of the usual Lie bracket of vector fields on the space $\mathcal{V}(M)$ \cite{damianou:duf-zung, damianou:kz}, is related to  the operator $i$ through  the following useful formula, due to Koszul \cite{damianou:kz}. For any $P\in \mathcal{V}^p(M)$ and $Q\in \mathcal{V}^q(M)$,
\begin{equation}\label{damianou:koszul-formula}
i_{[P,Q]} = [[i_P,d],i_Q],
\end{equation}
where the brackets on the right hand side of (\ref{damianou:koszul-formula}) denote the graded commutator of graded endomorphisms of $\Omega(M)$, i.e., for any two endomorphisms $E_1$ and $E_2$ of $\Omega(M)$ of degrees $e_1$ and $e_2$, respectively, $[E_1,E_2] = E_1\circ E_2 - (-1)^{e_1e_2}E_2\circ E_1$. Hence, we have
\begin{eqnarray}\label{damianou:koszul-formula-2}
\lefteqn{i_{[P,Q]} =  i_P\circ d \circ i_Q - (-1)^p\,d\circ i_P\circ i_Q} \nonumber \\
& & - (-1)^{(p-1)q}\,i_Q\circ i_P\circ d + (-1)^{(p-1)q-p}\,i_Q\circ d \circ i_P.
\end{eqnarray}

\vspace{2mm}

Furthermore, given a smooth \emph{volume form} $\Omega$ on $M$, i.e., a nowhere vanishing element of $\Omega^m(M)$, the interior product of $p$-vector fields on $M$, $p=0,1,\ldots,m$, with $\Omega$ yields a $C^{\infty}(M)$-linear isomorphism $\Psi$ of $\mathcal{V}(M)$ onto $\Omega(M)$ such that, for each degree $p$, $0\leq p \leq m$,
\begin{eqnarray*}
\Psi : \mathcal{V}^p(M) & \to & \Omega^{m-p}(M) \\
P & \mapsto & \Psi(P) = \Psi_P = (-1)^{(p-1)p/2}i_P\Omega.
\end{eqnarray*}
Its inverse map $\Psi^{-1}: \Omega^{m-p}(M) \to \mathcal{V}^p(M)$ is defined, for any $\eta \in \Omega^{m-p}(M)$, by $\Psi^{-1}({\eta}) = j_{\eta} \tilde{\Omega}$, where $\tilde{\Omega}$ denotes the dual $m$-vector field of $\Omega$, i.e., $\langle \Omega, \tilde{\Omega}\rangle =1$. By composing $\Psi$ with the exterior derivative $d$ on $\Omega(M)$ and $\Psi^{-1}$, we obtain the operator $D=-\Psi^{-1}\circ d \circ \Psi$ which was introduced by Koszul \cite{damianou:kz}. It is of degree $-1$ and of square $0$ and it generates the Schouten bracket. For any $P\in \mathcal{V}^p(M)$ and $Q\in \mathcal{V}(M)$,
\begin{equation}\label{damianou:schouten-D}
[P,Q] = (-1)^p\big(D(P\wedge Q)-D(P)\wedge Q - (-1)^pP\wedge D(Q) \big).
\end{equation}

\subsection{Poisson manifolds}
We recall the notion of \emph{Poisson manifold} and some of its properties whose proofs may be found, for example, in the books \cite{damianou:lm, damianou:duf-zung, damianou:vai-b}.

A \emph{Poisson structure} on a smooth manifold $M$ is a Lie algebra structure on $C^{\infty}(M)$ whose the bracket $\{\cdot,\cdot\} : C^{\infty}(M) \times C^{\infty}(M) \to C^{\infty}(M)$ verifies the Leibniz's rule:
\begin{equation*}
\{f,gh\} = \{f,g\}h + g\{f,h\}, \quad \quad \forall \, f,g,h \in C^{\infty}(M).
\end{equation*}
In \cite{damianou:lch1}, Lichn\'erowicz remarks that $\{\cdot,\cdot\}$ gives rise to a contravariant antisymmetric tensor field $\Lambda$  of order $2$ such that $\Lambda(df,dg) = \{f,g\}$, for $f,g\in C^{\infty}(M)$. Conversely, each  such bivector field $\Lambda$ on $M$ gives rise to a bilinear and antisymmetric bracket $\{\cdot,\cdot\}$ on $C^{\infty}(M)$, $\{f,g\} = \Lambda(df,dg)$, $f,g\in C^{\infty}(M)$, which  satisfies the Jacobi identity, i.e., for any $f,g,h \in C^{\infty}(M)$, $\{f,\{g,h\}\} + \{g,\{h,f\}\} + \{h,\{f,g\}\} = 0$, if and only if $[\Lambda,\Lambda]=0$, where $[\cdot , \cdot]$ denotes the Schouten bracket on $\mathcal{V}(M)$. In this case $\Lambda$ is called a \emph{Poisson tensor} and the manifold $(M,\Lambda)$ a \emph{Poisson manifold}. While, in the case where $[\Lambda,\Lambda]\neq 0$ we say that $\Lambda$ is an \emph{almost Poisson tensor}.

As was proved in \cite{damianou:Grab93}, it is a consequence of expression (\ref{damianou:koszul-formula-2}) of $[\cdot,\cdot]$  that an element $\Lambda \in \mathcal{V}^2(M)$ defines a Poisson structure on $M$ if and only if
\begin{equation*}
2i_{\Lambda} d \Psi_{\Lambda} + d \Psi_{\Lambda \wedge \Lambda}=0.\footnote{Because we have adopted a different convention of sign for the interior product $i$ from that in \cite{damianou:Grab93}, our condition differs up to a sign  from the one in  \cite{damianou:Grab93}.}
\end{equation*}
Equivalently, using formula (\ref{damianou:schouten-D}) of $[\cdot, \cdot]$ and the fact that, for any $P\in \mathcal{V}^p(M)$,
\begin{equation*}
\Psi^{-1}\circ i_P = (-1)^{(p-1)p/2}P\wedge \Psi^{-1},
\end{equation*}
the last condition can be written as
\begin{equation}\label{damianou:cond-D-Lambda}
2 \Lambda \wedge D(\Lambda) = D(\Lambda \wedge \Lambda).
\end{equation}

Given a Poisson tensor $\Lambda$ on $M$, we can associate to it a natural homomorphism $\Lambda^\# : \Omega^1(M) \to \mathcal{V}^1(M)$, which maps each element $\alpha$ of $\Omega^1(M)$ to a unique vector field $\Lambda^\#(\alpha)$ such that, for any $\beta\in \Omega^1(M)$,
\begin{equation*}
\langle \alpha \wedge \beta, \Lambda\rangle = \langle \beta, \Lambda^\#(\alpha)\rangle = \Lambda(\alpha,\beta).
\end{equation*}
If $\alpha = df$, for some $f\in C^{\infty}(M)$, the vector field $\Lambda^\#(df)$ is called the \emph{hamiltonian vector field of $f$ with respect to $\Lambda$} and it is denoted by $X_f$. The image $\mathrm{Im}\Lambda^\#$ of $\Lambda^\#$ is a completely integrable distribution on $M$ and defines the \emph{symplectic foliation} of $(M,\Lambda)$ whose  space of first integrals is the space of \emph{Casimir functions of} $\Lambda$, i.e., the space of the functions $f\in C^{\infty}(M)$ which  are solutions of $\Lambda^\#(df) =0$.

Moreover, $\Lambda^\#$ can be extended to a homomorphism, also denoted by $\Lambda^\#$,  from $\Omega^p(M)$ to $\mathcal{V}^p(M)$, $p\in \mathbb{N}$, by setting, for any $f\in C^{\infty}(M)$, $\Lambda^\#(f) = f$, and, for any $\zeta \in \Omega^p(M)$ and $\alpha_1,\ldots,\alpha_p\in \Omega^1(M)$,
\begin{equation}\label{damianou:def-extension}
\Lambda^\#(\zeta)(\alpha_1,\ldots,\alpha_p) = (-1)^p\zeta(\Lambda^\#(\alpha_1),\ldots,\Lambda^\#(\alpha_p)).
\end{equation}
Thus, $\Lambda^\#(\zeta \wedge \eta) = \Lambda^\#(\zeta) \wedge \Lambda^\#(\eta)$, for all $\eta\in \Omega(M)$. When $\Omega(M)$ is equipped with the Koszul bracket $\{\!\! \{ \cdot, \cdot \}\!\! \}$ defined, for any $\zeta \in \Omega^p(M)$ and $\eta \in \Omega(M)$, by
\begin{equation}\label{damianou:bracket-forms}
\{\!\! \{\zeta, \eta \}\!\! \} = (-1)^p \big(\Delta(\zeta \wedge \eta) - \Delta(\zeta)\wedge \eta - (-1)^p\zeta \wedge \Delta(\eta) \big),
\end{equation}
where $\Delta = i_{\Lambda}\circ d - d\circ i_{\Lambda}$, $\Lambda^\#$ becomes a graded Lie algebras homomorphism. Explicitly,
\begin{equation*}
\Lambda^\#(\{\!\! \{ \zeta, \eta \}\!\! \}) = [\Lambda^\#(\zeta),\Lambda^\#(\eta)] \ ,
\end{equation*}
where the bracket on the right hand side is the  Schouten bracket.

\begin{example}\label{example-Poisson}
{\rm Any symplectic manifold $(M,\omega_0)$, where  $\omega_0$ is a nondegenerate closed smooth $2$-form on $M$, is equipped with a Poisson structure $\Lambda_0$ defined by $\omega_0$ as follows. Define the tensor field $\Lambda_0$ to be the image of $\omega_0$ by the extension of the isomorphism $\Lambda_0^\# : \Omega^1(M) \to \mathcal{V}^1(M)$, (inverse of $\omega_0^\flat : \mathcal{V}^1(M) \to \Omega^1(M)$, $X\mapsto \omega_0^\flat (X) = - \omega_0(X, \cdot)$), to $\Omega^2(M)$,  given by (\ref{damianou:def-extension}).}
\end{example}

\subsection{Decomposition theorem for exterior differential forms}
In this subsection, we begin by reviewing some important results concerning the decomposition theorem for exterior differential forms on almost symplectic manifolds. The  complete study  of these results can be found  in \cite{damianou:lm} and \cite{damianou:lib-th}.

Let $(M,\omega_0)$ be a $2n$-dimensional almost symplectic manifold, i.e., $\omega_0$ is a nondegenerate smooth $2$-form on $M$, $\Lambda_0$ the bivector field on $M$ associated with $\omega_0$ (see, Example \ref{example-Poisson}), $\Omega = \displaystyle{\frac{\omega_0^n}{n!}}$ the corresponding volume form on $M$, and $\tilde{\Omega} =\displaystyle{\frac{\Lambda_0^n}{n!}}$ the dual $2n$-vector field of $\Omega$. We define an isomorphism $\ast : \Omega^p(M) \to \Omega^{2n-p}(M)$ by setting, for any $\varphi \in \Omega^p(M)$,
\begin{equation}\label{def-ast}
\ast \, \varphi = (-1)^{(p-1)p/2}\, i_{\Lambda_0^\#(\varphi)}\frac{\omega_0^n}{n!} \ .
\end{equation}

\begin{remark}
{\rm In order to be in agreement with the convention of sign adopted in (\ref{damianou:sign-inter. prod}) for the interior product, we make a  sign convention for $\ast$ different from  the one given in \cite{damianou:lm}.}
\end{remark}

The $(2n-p)$-form $\ast \, \varphi$ is called the \emph{adjoint of $\varphi$ relative to $\omega_0$}. The isomorphism $\ast$ has the following properties:
\begin{enumerate}
\item[i)] $\ast \, \ast = Id$.
\item[ii)] For any $\varphi\in \Omega^p(M)$ and $\psi \in \Omega^q(M)$,
\begin{eqnarray}\label{damianou:property-ii}
\ast\,(\varphi \wedge \psi) & = & (-1)^{(p+q-1)(p+q)/2}\,i_{\Lambda_0^\#(\varphi)\wedge\Lambda_0^\#(\psi)}\frac{\omega_0^n}{n!} \nonumber \\
                            & = & (-1)^{(p-1)p/2}\,i_{\Lambda_0^\#(\varphi)}(\ast\, \psi) = (-1)^{pq + (q-1)q/2} i_{\Lambda_0^\#(\psi)}(\ast \, \varphi).
\end{eqnarray}
\item[iii)] For any $k\leq n$,
\begin{equation*}
\ast\, \frac{\omega_0^k}{k!} = \frac{\omega_0^{n-k}}{(n-k)!}.
\end{equation*}
\end{enumerate}

\begin{definition}
A smooth form $\psi \in \Omega(M)$ such that $i_{\Lambda_0}\psi = 0$ everywhere on $M$ is said to be \emph{effective}. On the other hand, a smooth form $\varphi$ on $M$ is said to be \emph{simple} if it can be written as
\begin{equation*}
\varphi = \psi \wedge \frac{\omega_0^k}{k!},
\end{equation*}
where $\psi$ is effective.
\end{definition}

\begin{proposition}
The adjoint of an effective differential form $\psi$ of degree $p\leq n$ is
\begin{equation*}
\ast \, \psi = (-1)^{p(p+1)/2}\, \psi \wedge \frac{\omega_0^{n-p}}{(n-p)!}.
\end{equation*}
The adjoint $\ast \, \varphi$ of a smooth $(p+2k)$-simple form $\varphi = \psi \wedge \displaystyle{\frac{\omega_0^k}{k!}}$ is
\begin{equation}\label{damianou:adjoint-simple}
\ast \,\varphi = (-1)^{p(p+1)/2}\, \psi \wedge \frac{\omega_0^{n-p-k}}{(n-p-k)!}.
\end{equation}
\end{proposition}

\begin{th-lep}
Every differential form $\varphi\in \Omega(M)$, of degree $p\leq n$, may be uniquely decomposed as the sum
\begin{equation*}
\varphi = \psi_p + \psi_{p-2}\wedge \omega_0 + \ldots + \psi_{p-2q}\wedge \frac{\omega_0^q}{q!},
\end{equation*}
with $q\leq [p/2]$ ($[p/2]$ being the largest integer less than or equal to $p/2$), where, for $s = 0, \ldots, q$, the differential forms $\psi_{p-2s}$ are effective and may be calculated from $\varphi$ by means of iteration of the operator $i_{\Lambda_0}$. Then, its adjoint $\ast \, \varphi$ may be uniquely written as the sum
\begin{equation*}
\ast\, \varphi = (-1)^{p(p+1)/2}\big(\psi_p - \psi_{p-2}\wedge \frac{\omega_0}{n-p+1} + \ldots + (-1)^q\frac{(n-p)!}{(n-p+q)!}\psi_{p-2q}\wedge \omega_0^q \big)\wedge \frac{\omega_0^{n-p}}{(n-p)!}.
\end{equation*}
\end{th-lep}

\vspace{2mm}
We continue by indicating the relation which links $\ast$ with $\Psi$ and its effect on Poisson structures. Since $\Lambda_0^\# : \Omega^p(M) \to \mathcal{V}^p(M)$, $p \in \mathbb{N}$, defined by (\ref{damianou:def-extension}), is an isomorphism, for any smooth $p$-vector field $P$ on $M$ there exists an unique $p$-form $\sigma_p \in \Omega^p(M)$ such that $P = \Lambda_0^\#(\sigma_p)$. So, it is clear that
\begin{equation}\label{damianou:Psi-sigma}
\Psi (P) = \ast \, \sigma_p.
\end{equation}
In particular, a bivector field $\Lambda$ on $(M,\omega_0)$ can be viewed as the image $\Lambda_0^\#(\sigma)$ of a $2$-form $\sigma$ on $M$ by the isomorphism $\Lambda_0^\#$. We want to establish the condition on $\sigma$ under which $\Lambda = \Lambda_0^\#(\sigma)$ is a Poisson tensor. For this reason, we consider the \emph{codifferential operator}
\begin{equation*}
\delta = \ast \, d\,\ast
\end{equation*}
introduced in \cite{damianou:lib-th}, which is of degree $-1$ and satisfies the relation $\delta^2 = 0$, and we prove:

\begin{lemma}
For any differential form $\zeta$ on $(M,\omega_0)$ of degree $p\leq n$,
\begin{equation}\label{damianou:Psi-1,ast}
\Psi^{-1}(\zeta) = \Lambda_0^\#(\ast\,\zeta).
\end{equation}
\end{lemma}
\begin{proof}
Let $\eta$ be a smooth $(2n-p)$-form on $M$. We have
\begin{eqnarray*}
\langle \eta,\, \Psi^{-1}(\zeta) \rangle & = & \langle \eta, \,j_{\zeta}\frac{\Lambda_0^n}{n!} \rangle  = \langle \eta \wedge \zeta, \, \frac{\Lambda_0^n}{n!}\rangle \\
& = & (-1)^{2n} \langle \frac{\omega_0^n}{n!},\, \Lambda_0^\#(\eta)\wedge \Lambda_0^\#(\zeta)\rangle = (-1)^{p(2n-p)}\langle \frac{\omega_0^n}{n!}, \, \Lambda_0^\#(\zeta)\wedge\Lambda_0^\#(\eta)  \rangle \\
& = & (-1)^{p(2n-p)}(-1)^{(p-1)p/2}\,\langle i_{\Lambda_0^\#(\zeta)}\frac{\omega_0^n}{n!}, \,\Lambda_0^\#(\eta)\rangle \\
& = & (-1)^{p(2n-p)}(-1)^{(p-1)p/2}(-1)^{2n-p}\langle \eta, \,\Lambda_0^\#(i_{\Lambda_0^\#(\zeta)}\frac{\omega_0^n}{n!}) \rangle \\
& = & (-1)^{(p-1)p/2}\,\langle \eta,\, \Lambda_0^\#(i_{\Lambda_0^\#(\zeta)}\frac{\omega_0^n}{n!}) \rangle = \langle \eta,\, \Lambda_0^\#(\ast\, \zeta)\rangle,
\end{eqnarray*}
whence (\ref{damianou:Psi-1,ast}) follows. (We remark that the number $p(2n-p)+(2n-p) = (2n-p)(p+1)$ is even for any $p\in \mathbb{N}$.)
\end{proof}

\begin{proposition}\label{damianou:theorem-cond-delta-sigma}
Using the same notation, $\Lambda = \Lambda_0^\#(\sigma)$ defines a Poisson structure on $(M,\omega_0)$ if and only if
\begin{equation}\label{damianou:cond-delta-sigma}
2\sigma \wedge \delta(\sigma) = \delta (\sigma \wedge \sigma).
\end{equation}
\end{proposition}
\begin{proof}
We have seen that $\Lambda$ is a Poisson tensor if and only if (\ref{damianou:cond-D-Lambda}) holds. But, in our case $\Lambda = \Lambda_0^\#(\sigma)$, so $\Lambda \wedge \Lambda = \Lambda_0^\#(\sigma \wedge \sigma)$, and $\Lambda_0^\#$ is an isomorphism. Therefore,
\begin{eqnarray*}
2 \Lambda \wedge D(\Lambda) = D(\Lambda \wedge \Lambda) & \Leftrightarrow & - 2 \Lambda \wedge \big((\Psi^{-1}\circ d \circ \Psi)(\Lambda) \big) = -(\Psi^{-1}\circ d \circ \Psi)(\Lambda \wedge \Lambda) \nonumber \\
& \stackrel{(\ref{damianou:Psi-sigma})}{\Leftrightarrow} & 2\Lambda_0^\#(\sigma) \wedge \big(\Psi^{-1}(d\,\ast \sigma) \big) = \Psi^{-1}\big(d\,\ast (\sigma \wedge \sigma)\big) \nonumber \\
& \stackrel{(\ref{damianou:Psi-1,ast})}{\Leftrightarrow} & 2\Lambda_0^\#(\sigma) \wedge \Lambda_0^\# (\ast \,d\,\ast (\sigma)) = \Lambda_0^\#(\ast \,d \, \ast (\sigma \wedge \sigma))\nonumber \\
& \Leftrightarrow & \Lambda_0^\# (2 \sigma \wedge \delta \sigma ) = \Lambda_0^\# \big(\delta(\sigma \wedge \sigma)\big) \nonumber \\
& \Leftrightarrow & 2\sigma \wedge \delta(\sigma) = \delta (\sigma \wedge \sigma),
\end{eqnarray*}
and we are done.
\end{proof}

\begin{remark}
{\rm Brylinski \cite{damianou:brl} observed that, when the manifold is symplectic, i.e., $d\omega_0 = 0$, $\delta$ is equal, up to sign, to $\Delta = i_{\Lambda_0}\circ d - d \circ i_{\Lambda_0}$. Then, in this framework, (\ref{damianou:cond-delta-sigma}) is equivalent to $\{\!\! \{ \sigma, \sigma\}\!\! \}_{_0} =0$,  ($\{\!\! \{ \cdot, \cdot \}\!\! \}_{_0}$ being the Koszul bracket (\ref{damianou:bracket-forms}) associated to $\Lambda_0$), which means that $\sigma$ is a complementary $2$-form on $(M,\Lambda_0)$ in the sense of Vaisman \cite{damianou:vai}.}
\end{remark}

\section{Poisson structures with prescribed Casimir functions}\label{damianou:section-theorem}
Let $M$ be a $m$-dimensional smooth manifold and $f_1,\ldots,f_{m-2k}$ smooth functions on $M$ which are functionally independent almost everywhere. We want to construct Poisson structures $\Lambda$ on $M$ with symplectic leaves of dimension at most $2k$ which have as Casimirs the given functions $f_1, f_2, \ldots, f_{m-2k}$. We start by discussing the problem on even-dimensional manifolds. In the next subsection we extend the results to  odd-dimensional manifolds.

\subsection{On even-dimensional manifolds}\label{section-even}
We suppose that $\dim M = 2n$ and we begin our study by remarking
\begin{lemma}
Given $(M,\,f_1,\ldots,f_{2n-2k})$, with $f_1,\ldots,f_{2n-2k}$ functionally independent almost everywhere on $M$, then there exists, at least locally, $\Lambda_0 \in \mathcal{V}^2(M)$, with $rank\,\Lambda_0=2n$, such that
\begin{equation*}
\big\langle df_1\wedge \ldots \wedge df_{2n-2k},\; \frac{\Lambda_0^{n-k}}{(n-k)!}\big\rangle \neq 0.
\end{equation*}
\end{lemma}
\begin{proof}
In fact, let $p \in M$ and $U$ an open neighborhood of $p$ such that $f_1,\ldots,f_{2n-2k}$ are functionally independent at each point $x\in U$. That means that $df_1\wedge \ldots \wedge df_{2n-2k}(x) \neq 0$ on $U$. We select $1$-forms $\beta_1,\ldots,\beta_{2k}$ on $U$ so that $(df_1, \ldots, df_{2n-2k},\beta_1,\ldots,\beta_{2k})$ become a basis of the cotangent space at each point. Let $(Y_1,\ldots,Y_{2n-2k},Z_1,\ldots,Z_{2k})$ be a family of vector fields on $U$ dual to $(df_1, \ldots, df_{2n-2k},\beta_1,\ldots,\beta_{2k})$. That is, they satisfy $df_i(Y_j) = \delta_{ij}$, $\beta_i(Z_j)=\delta_{ij}$, and all other pairings are zero. We consider the bivector field
\begin{equation*}
\Lambda_0 = \sum_{i=1}^{n-k}Y_{2i-1}\wedge Y_{2i} + \sum_{j=1}^kZ_{2j-1}\wedge Z_{2j}
\end{equation*}
which is of maximal rank on $U$. It is clear that
\begin{equation*}
\langle df_1\wedge \ldots \wedge df_{2n-2k}, \,\frac{\Lambda_0^{n-k}}{(n-k)!}\rangle = 1 \neq 0.
\end{equation*}
\end{proof}

\vspace{2mm}
Consider now $(M,\,f_1,\ldots,f_{2n-2k})$ and a nondegenerate bivector field $\Lambda_0$ on $M$ such that
\begin{equation}\label{damianou:f}
f = \big\langle df_1\wedge \ldots \wedge df_{2n-2k},\; \frac{\Lambda_0^{n-k}}{(n-k)!}\big\rangle =\big \langle \frac{\omega_0^{n-k}}{(n-k)!}, \;X_{f_1}\wedge \ldots \wedge X_{f_{2n-2k}}\big\rangle \neq 0
\end{equation}
on an open and dense subset $\mathcal{U}$ of $M$. In (\ref{damianou:f}), $\omega_0$ denotes the almost symplectic form on $M$ defined by $\Lambda_0$ and $X_{f_i} = \Lambda_0^\#(df_i)$ are the  hamiltonian  vector fields of $f_i$, $i = 1,\ldots, 2n-2k$, with respect to $\Lambda_0$. Let $D = \langle X_{f_1}, \ldots, X_{f_{2n-2k}}\rangle$ be the distribution on $M$ generated by $X_{f_i}$, $i=1,\ldots,2n-2k$, $D^\circ$ its annihilator, and $\mathrm{orth}_{\omega_0}D$ the symplectic orthogonal of $D$ with respect to $\omega_0$. Since $\det\big(\{f_i,f_j\}_{_0} \big) = f^2 \neq 0$ on $\mathcal{U}$, $D_x = D \cap T_xM$ is a symplectic subspace of $T_xM$ with respect to $\omega_{0_x}$ at each point $x\in \mathcal{U}$. Thus, $T_xM = D_x \oplus \mathrm{orth}_{\omega_{0_x}}D_x = D_x \oplus \Lambda_{0_x}^\#(D_x^\circ)$, where $D_x^\circ = D^\circ \cap T_x^\ast M$, and $T_x^\ast M = D_x^\circ \oplus (\Lambda_{0_x}^\#(D_x^\circ))^\circ = D_x^\circ \oplus \langle df_1,\ldots,df_{2n-2k}\rangle_x$. Finally, we denote by $\sigma$ the smooth $2$-form on $M$ which corresponds, via the isomorphism $\Lambda_0^\#$, to an element $\Lambda$ of $\mathcal{V}^2(M)$.

\begin{proposition}\label{damianou:prop-Lambda-sigma}
Under the above assumptions, a bivector field $\Lambda$ on $(M,\omega_0)$, of rank at most $2k$ on $M$, admits  as unique Casimirs the functions $f_1,\ldots, f_{2n-2k}$ if and only if its corresponding $2$-form $\sigma$ is a smooth section of $\bigwedge^2D^\circ$ of maximal rank on $\mathcal{U}$.
\end{proposition}
\begin{proof}
Effectively, for any $f_i$, $i=1,\ldots, 2n-2k$,
\begin{equation}\label{damianou:prop-Lambda-sigma-fi}
\Lambda(df_i, \cdot) = 0 \Leftrightarrow \Lambda_0^\#(\sigma)(df_i, \cdot) = 0 \Leftrightarrow \sigma (X_{f_i}, \Lambda_0^\#(\cdot)) = 0.
\end{equation}
Thus, $f_1,\ldots,f_{2n-2k}$ are the unique Casimir functions of $\Lambda$ on $\mathcal{U}$ if and only if the vector fields $X_{f_1},\ldots,X_{f_{2n-2k}}$ with functionally independent hamiltonians on $\mathcal{U}$ generate $\ker \sigma$, i.e., for any $x\in \mathcal{U}$, $D_x = \ker \sigma_x^\flat$. The last relation  means that $\sigma$ is a section of $\bigwedge^2D^\circ$ of maximal rank on $\mathcal{U}$.
\end{proof}

\vspace{3mm}

Still using  the same  notation, we can formulate the following main theorem.

\begin{theorem}\label{damianou:THEOREM}
Let $f_1,\ldots,f_{2n-2k}$ be smooth functions on a $2n$-dimensional differentiable manifold $M$ which are functionally independent almost everywhere, $\omega_0$ an almost symplectic structure on $M$ such that (\ref{damianou:f}) holds on an open and dense subset $\mathcal{U}$ of $M$, $\Omega = \displaystyle{\frac{\omega_0^n}{n!}}$ the corresponding volume form on $M$, and $\sigma$ a section of $\bigwedge^2 D^{\circ}$ of maximal rank on $\mathcal{U}$ that satisfies (\ref{damianou:cond-delta-sigma}). Then, the $(2n-2)$-form
\begin{equation}\label{damianou:expression-Phi}
\Phi = - \frac{1}{f}(\sigma + \frac{g}{k-1}\omega_0)\wedge \frac{\omega_0^{k-2}}{(k-2)!}\wedge df_1\wedge\ldots \wedge df_{2n-2k},
\end{equation}
where $f$ is given by (\ref{damianou:f}) and $g = i_{\Lambda_0}\sigma$, corresponds, via the isomorphism $\Psi^{-1}$, to a Poisson tensor $\Lambda$ with orbits of dimension at most $2k$ for which $f_1,\ldots,f_{2n-2k}$ are Casimirs. Precisely, $\Lambda = \Lambda_0^\#(\sigma)$ and the associated bracket of $\Lambda$ on $C^{\infty}(M)$ is given, for any $h_1,h_2 \in C^{\infty}(M)$, by
\begin{equation}\label{damianou:bracket-Lambda-Omega}
\{h_1,h_2\} \Omega =  -\frac{1}{f} dh_1 \wedge dh_2 \wedge (\sigma + \frac{g}{k-1}\omega_0) \wedge \frac{\omega_0^{k-2}}{(k-2)!}\wedge df_1\wedge\ldots \wedge df_{2n-2k}.
\end{equation}
Conversely, if $\Lambda \in \mathcal{V}^2(M)$ is a Poisson tensor of rank $2k$ on an open and dense subset $\mathcal{U}$ of $M$, then there are $2n-2k$ functionally independent smooth functions $f_1,\ldots,f_{2n-2k}$ on $\mathcal{U}$ and a section $\sigma$ of $\bigwedge^2 D^{\circ}$ of maximal rank on $\mathcal{U}$ satisfying (\ref{damianou:cond-delta-sigma}), such that $\Psi_{\Lambda}$ and $\{\cdot,\cdot\}$ are given, respectively, by (\ref{damianou:expression-Phi}) and (\ref{damianou:bracket-Lambda-Omega}).
\end{theorem}
\begin{proof}
We denote by $\tilde{\Omega}= \displaystyle{\frac{\Lambda_0^n}{n!}}$ the dual $2n$-vector field of $\Omega$ on $M$ and we set $\Lambda = j_{\Phi}\tilde{\Omega}$. For any $f_i$, $i = 1,\ldots,2n-2k$, we have
\begin{equation*}
\Lambda^\#(df_i) =- j_{df_i}\Lambda = -j_{df_i}j_{\Phi}\tilde{\Omega} = -j_{df_i\wedge \Phi}\tilde{\Omega} = -j_0\tilde{\Omega} = 0,
\end{equation*}
which means that $f_1,\ldots, f_{2n-2k}$ are Casimir functions of $\Lambda$. We shall see that $\Lambda = \Lambda_0^\#(\sigma)$. Thus, $\Lambda$ will define a Poisson structure on $M$ having the required properties. We calculate the adjoint form $\ast\, \Phi$ of $\Phi$ relative to $\omega_0$:
\begin{eqnarray*}
\ast\,\Phi & = & - \frac{1}{f} \ast\big((\sigma + \frac{g}{k-1}\omega_0)\wedge \frac{\omega_0^{k-2}}{(k-2)!}\wedge df_1 \ldots \wedge df_{2n-2k}\big) \nonumber \\
& \stackrel{(\ref{damianou:property-ii})}{=} & - (-1)^{(2n-2k-1)(2n-2k)/2}\,\frac{1}{f}i_{X_{f_1}\wedge \ldots \wedge X_{f_{2n-2k}}} \big[\ast \big((\sigma + \frac{g}{k-1}\omega_0)\wedge \frac{\omega_0^{k-2}}{(k-2)!}\big) \big].
\end{eqnarray*}
But, from Lepage's decomposition theorem, $\sigma$ can be written as $\sigma = \psi_2 + \psi_0\omega_0$, where $\psi_2$ is an effective $2$-form on $M$ with respect to $\Lambda_0$ and $\psi_0 = \displaystyle{\frac{i_{\Lambda_0}\sigma}{i_{\Lambda_0}\omega_0} = -\frac{g}{n}}$. (It is easy to check that $i_{\Lambda_0}\omega_0 = - \langle \omega_0,\Lambda_0\rangle = -\displaystyle{ \frac{Tr(\omega_0^\flat \circ \Lambda_0^\#)}{2}}= - \frac{Tr(I_{2n})}{2}=-n.$) Hence,
\begin{eqnarray*}
(\sigma +\frac{g}{k-1}\omega_0)\wedge \frac{\omega_0^{k-2}}{(k-2)!} & = & (\psi_2 - \frac{g}{n}\omega_0 + \frac{g}{k-1}\omega_0)\wedge \frac{\omega_0^{k-2}}{(k-2)!} \\
& = & \psi_2\wedge \frac{\omega_0^{k-2}}{(k-2)!} + \frac{n-k+1}{n}g\frac{\omega_0^{k-1}}{(k-1)!}
\end{eqnarray*}
and
\begin{eqnarray}\label{damianou:ast-sigma-k-1}
\ast \big((\sigma + \frac{g}{k-1}\omega_0)\wedge \frac{\omega_0^{k-2}}{(k-2)!}\big) & = & \ast\, \big(\psi_2\wedge \frac{\omega_0^{k-2}}{(k-2)!}\big) + \frac{n-k+1}{n}g \big(\ast\,\frac{\omega_0^{k-1}}{(k-1)!}\big) \nonumber \\
& \stackrel{(\ref{damianou:adjoint-simple})}{=} & - \psi_2 \wedge \frac{\omega_0^{n-(k-2)-2}}{(n-(k-2)-2)!} + \frac{n-k+1}{n}g \frac{\omega_0^{n-(k-1)}}{(n-(k-1))!} \nonumber \\
& =& - (\psi_2 - \frac{g}{n}\omega_0)\wedge \frac{\omega_0^{n-k}}{(n-k)!}\, = \,-\sigma \wedge \frac{\omega_0^{n-k}}{(n-k)!}.
\end{eqnarray}
Consequently,
\begin{eqnarray}\label{damianou:ast-Phi-sigma}
\ast\, \Phi & = & - (-1)^{(2n-2k-1)(2n-2k)/2}\,\frac{1}{f}i_{X_{f_1}\wedge \ldots \wedge X_{f_{2n-2k}}} \big[ -\sigma \wedge \frac{\omega_0^{n-k}}{(n-k)!}\big] \nonumber \\
& \stackrel{(\ref{damianou:sign-inter. prod})(\ref{damianou:prop-Lambda-sigma-fi})}{=} & \frac{1}{f} \big \langle \frac{\omega_0^{n-k}}{(n-k)!},\, X_{f_1}\wedge \ldots \wedge X_{f_{2n-2k}}\big \rangle \, \sigma \,= \frac{1}{f}f\,\sigma = \sigma.
\end{eqnarray}
By applying (\ref{damianou:Psi-1,ast}) to the above relation, we obtain
\begin{equation*}
\Lambda_0^\#(\sigma) = \Lambda_0^\#(\ast\,\Phi) = \Psi^{-1}(\Phi) = j_\Phi \tilde{\Omega} = \Lambda.
\end{equation*}
Thus, according to Proposition \ref{damianou:theorem-cond-delta-sigma}, $\Lambda$ defines a Poisson structure on $M$, with orbits of dimension at most $2k$, for which $f_1, \ldots, f_{2n-2k}$ are Casimir functions. Obviously, the associated bracket of $\Lambda$ on $C^{\infty}(M)$ is given by (\ref{damianou:bracket-Lambda-Omega}). For any $h_1,h_2 \in C^{\infty}(M)$,
\begin{eqnarray*}
\{h_1,h_2\}& = & j_{dh_1 \wedge dh_2}\Lambda = j_{dh_1 \wedge dh_2}j_\Phi \tilde{\Omega} = j_{dh_1\wedge dh_2 \wedge \Phi}\tilde{\Omega} \;\;\; \Leftrightarrow \nonumber \\
\{h_1,h_2\}\Omega &= & -\frac{1}{f}dh_1\wedge dh_2 \wedge (\sigma + \frac{g}{k-1}\omega_0) \wedge \frac{\omega_0^{k-2}}{(k-2)!}\wedge df_1\wedge\ldots \wedge df_{2n-2k}.
\end{eqnarray*}

Conversely, if $\Lambda$ is a Poisson tensor on $M$ with symplectic leaves of dimension at most $2k$, then in a neighborhood $U$ of a nonsingular point there are coordinates $(z_1,\ldots,z_{2k},f_1,\ldots,f_{2n-2k})$ such that the symplectic leaves of $\Lambda$ are defined by $f_l = const.$, $l= 1,\ldots, 2n-2k$. Let $\Lambda_0$ be a nondegenerate bivector field on $U$ such that $f=\langle df_1\wedge \ldots \wedge df_{2n-2k},\: \displaystyle{\frac{\Lambda_0^{n-k}}{(n-k)!}} \rangle \neq 0$ on $U$ and $\sigma$ the $2$-form on $U$ which corresponds, via the isomorphism $\Lambda_0^\#$, to $\Lambda$. As we did earlier, we construct the distribution $D$ on $U$ and its annihilator $D^\circ$. According to Propositions \ref{damianou:prop-Lambda-sigma} and \ref{damianou:theorem-cond-delta-sigma}, $\sigma$ is a section of $\bigwedge^2 D^\circ$ of maximal rank on $U$ satisfying (\ref{damianou:cond-delta-sigma}). We will prove that the $(2n-2)$-form $\Psi_{\Lambda} = -i_{\Lambda_0^\#(\sigma)}\Omega = \ast\,\sigma$, where $\Omega = \displaystyle{\frac{\omega_0^n}{n!}}$ is the volume element on $U$ defined by the almost symplectic form $\omega_0$, the inverse of $\Lambda_0$, can be written in the form (\ref{damianou:expression-Phi}).

Since (\ref{damianou:f}) holds on $U$, $\Omega$ can be written on $U$ as
\begin{equation*}
\Omega = \frac{1}{f}\frac{\omega_0^k}{k!}\wedge df_1\wedge \ldots \wedge df_{2n-2k}
\end{equation*}
and
\begin{equation}\label{damianou:eq-Lambda-fi}
\Psi_{\Lambda} =  - i_{\Lambda}\Omega = - \frac{1}{f}\big(i_{\Lambda} \frac{\omega_0^k}{k!}\big) \wedge df_1 \wedge \ldots \wedge df_{2n-2k}.
\end{equation}
We now proceed  to calculate the $(2k-2)$-form $\displaystyle{-i_{\Lambda}\frac{\omega_0^k}{k!}}$. We remark that $\displaystyle{\frac{\omega_0^k}{k!} = \ast \, \frac{\omega_0^{n-k}}{(n-k)!}}$. So, from (\ref{damianou:property-ii}) we get that
\begin{equation}\label{damianou:eq-Lambda-k}
-i_{\Lambda}\frac{\omega_0^k}{k!} = \ast \, (\sigma \wedge \frac{\omega_0^{n-k}}{(n-k)!}).
\end{equation}
Repeating  the calculation of (\ref{damianou:ast-sigma-k-1}) in the inverse direction, we have
\begin{equation}\label{damianou:ast-sigma-k}
\ast \, (\sigma \wedge \frac{\omega_0^{n-k}}{(n-k)!}) = - \ast \ast \big((\sigma + \frac{g}{k-1}\omega_0)\wedge \frac{\omega_0^{k-2}}{(k-2)!}\big) = - (\sigma + \frac{g}{k-1}\omega_0)\wedge \frac{\omega_0^{k-2}}{(k-2)!}.
\end{equation}
Therefore, by replacing (\ref{damianou:ast-sigma-k}) in (\ref{damianou:eq-Lambda-k}) and the obtained relation in (\ref{damianou:eq-Lambda-fi}), we prove that $\Psi_{\Lambda}$ is given by   the expression (\ref{damianou:expression-Phi}). Then, it is clear  that $\{\cdot,\cdot\}$ is given by (\ref{damianou:bracket-Lambda-Omega}).
\end{proof}

\vspace{3mm}

\noindent
\textbf{The case of almost Poisson structures with prescribed kernel.} Theorem \ref{damianou:THEOREM} can be generalized by replacing the exact $1$-forms $df_1,\ldots,df_{2n-2k}$ with $1$-forms $\alpha_1,\ldots,\alpha_{2n-2k}$ which are linearly independent at each point of an open and dense subset of $M$. It suffices to consider a nondegenerate almost Poisson structure $\Lambda_0$ on $M$ such that
\begin{equation*}
f = \big\langle \alpha_1\wedge \ldots \wedge \alpha_{2n-2k},\; \frac{\Lambda_0^{n-k}}{(n-k)!}\big\rangle \neq 0
\end{equation*}
holds on an open and dense subset $\mathcal{U}$ of $M$ and to construct the distribution $D = \langle X_{\alpha_1},\ldots,X_{\alpha_{2n-2k}}\rangle$, $X_{\alpha_i} = \Lambda_0^\#(\alpha_i)$, and its annihilator $D^\circ$. Then, to each section $\sigma$ of $\bigwedge^2 D^\circ$ of maximal rank on $\mathcal{U}$ corresponds an almost Poisson structure $\Lambda \in \mathcal{V}^2(M)$ of rank at most $2k$ whose  kernel coincides with the space $\langle \alpha_1,\ldots,\alpha_{2n-2k}\rangle$ almost everywhere on $M$ and its associated bracket on $\C$ is given by
\begin{equation}\label{br-almostPoisson}
\{h_1,h_2\} \Omega = - \frac{1}{f}dh_1\wedge dh_2\wedge(\sigma + \frac{g}{k-1}\omega_0)\wedge \frac{\omega_0^{k-2}}{(k-2)!}\wedge \alpha_1\wedge\ldots \wedge \alpha_{2n-2k},
\end{equation}
$\omega_0$ being the almost symplectic structure on $M$ defined by $\Lambda_0$, $g = i_{\Lambda_0}\sigma$ and $\Omega = \displaystyle{\frac{\omega_0^n}{n!}}$.

\subsection{On odd-dimensional manifolds}
Let $M$ be a $(2n+1)$-dimensional manifold. We remark that any Poisson tensor $\Lambda$ on $M$ admitting  $f_1,\ldots,f_{2n+1-2k} \in \C$ as Casimir functions can be viewed as a  Poisson tensor on $M'=M\times \R$ admitting  $f_1,\ldots,f_{2n+1-2k}$ and $f_{2n+2-2k}(x,s)=s$   ( $s$ being the canonical coordinate on the factor $\R$) as Casimir functions, and conversely. Thus, the problem of construction of Poisson brackets on $\C$ having as  center the space of functions generated by $(f_1,\ldots,f_{2n+1-2k})$ is equivalent to that of construction of Poisson brackets on $C^\infty(M')$ having as  center the space of functions generated by $(f_1,\ldots,f_{2n+1-2k}, s)$, a setting which was completely studied in subsection \ref{section-even}. In what follows, using the results of 3.1., we establish a formula analogous to (\ref{damianou:bracket-Lambda-Omega}) for Poisson brackets on odd-dimensional manifolds. But, before we proceed, let us recall the notion of \emph{almost cosymplectic} structures on $M$ and some of their properties \cite{lib1, lch2}.

\vspace{2mm}

An \emph{almost cosymplectic} structure on a smooth manifold $M$, with $\dim M = 2n+1$, is defined by a pair $(\vartheta_0,\Theta_0)\in \Omega^1(M)\times \Omega^2(M)$ such that $\vartheta_0 \wedge \Theta_0^n \neq 0$ everywhere on $M$. The last condition means that $\vartheta_0 \wedge \Theta_0^n$ is a volume form on $M$ and that $\Theta_0$ is of constant rank $2n$ on $M$. Thus, $\ker \vartheta_0$ and $\ker \Theta_0$ are complementary subbundles of $TM$ called, respectively, the \emph{horizontal bundle} and the \emph{vertical bundle}. Of course, their annihilators are complementery subbundles of $T^\ast M$. Moreover, it is well known \cite{lch2} that $(\vartheta_0,\Theta_0)$ gives rice to a transitive \emph{almost Jacobi} structure $(\Lambda_0,E_0) \in \mathcal{V}^2(M)\times \mathcal{V}^1(M)$ on $M$ such that
\begin{equation*}
i(E_0)\vartheta_0 = 1 \quad  \mathrm{and}  \quad i(E_0)\Theta_0 = 0,
\end{equation*}
\begin{equation*}
\Lambda_0^\#(\vartheta_0) = 0 \quad \mathrm{and} \quad i(\Lambda_0^\#(\zeta)\Theta_0) = -(\zeta - \langle \zeta, E_0\rangle \vartheta_0), \quad \mathrm{for}\;\;\mathrm{all}\;\; \zeta \in \Omega^1(M).
\end{equation*}
We have, $\ker \vartheta_0 = \mathrm{Im}\Lambda_0^\#$ and $\ker \Theta_0 = \langle E_0\rangle$. So, $TM =\mathrm{Im}\Lambda_0^\# \oplus \langle E_0\rangle$ and $T^\ast M = \langle E_0\rangle^\circ \oplus \langle \vartheta_0\rangle$. The sections of $\langle E_0\rangle^\circ$ are called \emph{semi-basic} forms and $\Lambda_0^\#$ is an isomorphism from the $\C$-module of semi-basic $1$-forms to  the $\C$-module of horizontal vector fields. This isomorphism can be  extended, as in (\ref{damianou:def-extension}), to an isomorphism, also denoted by $\Lambda_0^\#$, from the $\C$-module of semi-basic $p$-forms on the $\C$-module of horizontal $p$-vector fields. Finally, we note that $(\vartheta_0,\Theta_0)$ determines on $M'=M\times \R$ an almost symplectic structure $\omega'_0 = \Theta_0 + ds \wedge \vartheta_0$ whose  corresponding nondegenerate almost Poisson tensor is $\Lambda'_0 = \Lambda_0 + \displaystyle{\frac{\partial}{\partial s}}\wedge E_0$.

\vspace{2mm}

Now, we consider $(M, f_1,\ldots,f_{2n+1-2k})$, with $f_1,\ldots,f_{2n+1-2k}$ functionally independent almost everywhere on $M$, and an almost cosymplectic structure $(\vartheta_0,\Theta_0)$ on $M$ whose associated nondegenerate almost Jacobi structure $(\Lambda_0, E_0)$ verifies the condition
\begin{equation}\label{f-odd}
f = \langle df_1\wedge \ldots \wedge df_{2n+1-2k},\; E_0\wedge \frac{\Lambda_0^{n-k}}{(n-k)!}\rangle \neq 0
\end{equation}
on an open and dense subset $\mathcal{U}$ of $M$.\footnote{As in the case of even-dimensional manifolds,  such a structure $(\Lambda_0,E_0)$ always exists at least locally.} Let $\omega_0'= \Theta_0 + ds \wedge \vartheta_0$ and $\Lambda'_0 = \Lambda_0 + \displaystyle{\frac{\vartheta}{\vartheta s}}\wedge E_0$ be, respectively, the associated almost symplectic and almost Poisson structure on $M'=M\times \R$. Since, for any $m = 1,\ldots,n+1$,
\begin{equation}\label{volume'}
\frac{\omega_0'\,^m}{m!} = \frac{\Theta_0^m}{m!} + ds\wedge \vartheta_0 \wedge \frac{\Theta_0^{m-1}}{(m-1)!} \quad \mathrm{and} \quad \frac{\Lambda_0'\,^m}{m!} = \frac{\Lambda_0^m}{m!} + \frac{\partial}{\partial s}\wedge E_0 \wedge \frac{\Lambda_0^{m-1}}{(m-1)!},
\end{equation}
it is clear that
\begin{eqnarray}\label{f-odd-even}
\lefteqn{\langle df_1\wedge \ldots \wedge df_{2n+1-2k}\wedge ds,\; \frac{\Lambda_0'\,^{n+1-k}}{(n+1-k)!} \rangle }\nonumber \\
& = & \langle df_1\wedge \ldots \wedge df_{2n+1-2k}\wedge ds,\; \frac{\Lambda_0^{n+1-k}}{(n+1-k)!} + \frac{\partial}{\partial s}\wedge E_0 \wedge \frac{\Lambda_0^{n-k}}{(n-k)!}\rangle \nonumber \\
& = & \langle df_1\wedge \ldots \wedge df_{2n+1-2k}\wedge ds,\; \frac{\partial}{\partial s}\wedge E_0 \wedge \frac{\Lambda_0^{n-k}}{(n-k)!}\rangle = - f \neq 0
\end{eqnarray}
on the open and dense subset $\mathcal{U}'= \mathcal{U} \times \R$ of $M'$. Furthermore, we view any bivector field $\Lambda$ on $(M,\vartheta_0,\Theta_0)$ having as  Casimirs the given functions as a bivector field on $(M',\omega_0')$ having $f_1,\ldots,f_{2n+1-2k}$ and $f_{2n+2-2k}(x,s)=s$ as Casimirs. Let $D'^\circ$ be the annihilator of the distribution $D'=\langle X_{f_1}',\ldots,X_{f_{2n+2-2k}}' \rangle$ on $M'$ generated by the hamiltonian vector fields $X_{f_i}'= \Lambda_0'^{\#}(df_i) = \Lambda_0^\#(df_i)-\langle df_i,E_0\rangle \displaystyle{\frac{\partial}{\partial s}}$, $i = 1,\ldots,2n+1-2k$, and $X'_{f_{2n+2-2k}} = \Lambda_0'^{\#}(ds)=E_0$ of $f_1,\ldots,f_{2n+1-2k}$ and $f_{2n+2-2k}(x,s) = s$ with respect to $\Lambda_0'$. Then, from Proposition \ref{damianou:prop-Lambda-sigma} we get that there exists an unique $2$-form $\sigma'$ on $M'$, which is a section of $\bigwedge^2 D'^\circ$ of maximal rank $2k$ on $\mathcal{U}'=\mathcal{U}\times \R$, such that $\Lambda = \Lambda_0'^{\#}(\sigma')$. Moreover, since  $\Lambda$ is independent of $s$ and without  a term of type $X\wedge \displaystyle{\frac{\partial}{\partial s}}$, $\sigma'$ must be of type
\begin{equation}\label{sigma'}
\sigma' = \sigma + \tau \wedge ds,
\end{equation}
where $\sigma$ and $\tau$ are, respectively, a $2$-form and a $1$-form on $M$ having the following additional properties:
\begin{enumerate}
\item [i)]
$\sigma$ is a section $\bigwedge^2 \langle E_0\rangle^\circ$, i.e., $\sigma$ is a semi-basic $2$-form on $M$ with respect to $(\Lambda_0,E_0)$;
\item[ii)]
$\tau$ is a section of $D^\circ = \langle X_{f_1},\ldots,X_{f_{2n+1-2k}},E_0\rangle^\circ$, where $X_{f_i} = \Lambda_0^\#(df_i)$, i.e., $\tau$ is a semi-basic $1$-form on $(M,\Lambda_0,E_0)$ which is also semi-basic with respect to $X_{f_1},\ldots, X_{f_{2n+1-2k}}$;
\item[iii)]
for any $f_i$, $i=1,\ldots,2n+1-2k$, $\sigma(X_{f_i},\cdot) + \langle df_i,E_0\rangle \tau =0$.
\end{enumerate}
Consequently, $\Lambda$ is written, in an unique way, as $\Lambda = \Lambda_0^\#(\sigma) + \Lambda_0^\#(\tau)\wedge E_0$.

\vspace{1mm}

Summarizing, we may formulate the next Proposition.
\begin{proposition}
Under the above notations and assumptions, a bivector field $\Lambda$ on $(M,\vartheta_0,\Theta_0)$, of rank at most $2k$, has as unique Casimirs the functions $f_1,\ldots, f_{2n+1-2k}$ if and only if its corresponding pair of forms $(\sigma,\tau)$ has the properties (i)-(iii) and $(rank\,\sigma,\, rank\,\tau)=(2k,0)$ or $(2k,1)$ or $(2k-2,1)$ on $\mathcal{U}$.
\end{proposition}

On the other hand, it follows from Theorem \ref{damianou:THEOREM} that the bracket $\{\cdot,\cdot\}$ of $\Lambda$ on $\C$ is calculated, for any $h_1,h_2\in \C$, viewed as elements of $C^\infty(M')$, by the formula
\begin{equation*}
\{h_1,h_2\}\Omega'\stackrel{(\ref{f-odd-even})}{=} \frac{1}{f}dh_1\wedge dh_2\wedge (\sigma'+ \frac{g'}{k-1}\omega_0')\wedge \frac{\omega_0'\,^{k-2}}{(k-2)!}\wedge df_1\wedge \ldots \wedge df_{2n+1-2k}\wedge ds,
\end{equation*}
where $\Omega'=\displaystyle{\frac{\omega_0'\,^{n+1}}{(n+1)!}}$ and $g'=i_{\Lambda_0'}\sigma'$. But, $\Omega'\stackrel{(\ref{volume'})}{=}-\Omega\wedge ds$, $\Omega = \vartheta_0\wedge \displaystyle{\frac{\Theta_0^n}{n!}}$ being a volume form on $M$, and $g'=i_{\Lambda_0'}\sigma' = i_{\Lambda_0 + \partial/\partial s\wedge E_0}(\sigma + \tau \wedge ds) = i_{\Lambda_0}\sigma = g$. Thus, taking into account (\ref{volume'}) and (\ref{sigma'}), we have
\begin{equation*}
\{h_1,h_2\}\Omega\wedge ds = - \frac{1}{f}dh_1\wedge dh_2\wedge (\sigma+ \frac{g}{k-1}\Theta_0)\wedge \frac{\Theta_0^{k-2}}{(k-2)!}\wedge df_1\wedge \ldots \wedge df_{2n+1-2k}\wedge ds
\end{equation*}
which is equivalent to
\begin{equation*}
\{h_1,h_2\}\Omega = - \frac{1}{f}dh_1\wedge dh_2\wedge (\sigma+ \frac{g}{k-1}\Theta_0)\wedge \frac{\Theta_0^{k-2}}{(k-2)!}\wedge df_1\wedge \ldots \wedge df_{2n+1-2k}.
\end{equation*}

However, according to Proposition \ref{damianou:theorem-cond-delta-sigma}, $\{\cdot,\cdot\}$ is a Poisson bracket on $\C \subset C^\infty(M')$ if and only if
\begin{equation}\label{cond-sigma'}
2\sigma'\wedge \delta'(\sigma')=\delta'(\sigma'\wedge \sigma'),
\end{equation}
where $\delta'= \ast\,'d\,\ast'$ is the codifferential on $\Omega(M')$ of $(M',\omega_0')$ defined by the isomorphism $\ast': \Omega^p(M') \to \Omega^{2n+2-p}(M')$ of (\ref{def-ast}). We want to translate (\ref{cond-sigma'}) to a condition on $(\sigma,\tau)$. Let $\Omega_{sb}^p(M)$ be the space of semi-basic $p$-forms on $(M,\Lambda_0,E_0)$, $\ast$ the isomorphism between $\Omega_{sb}^p(M)$ and $\Omega_{sb}^{2n-p}(M)$ given, for any $\varphi \in \Omega_{sb}^p(M)$, by
\begin{equation*}
\ast \, \varphi = (-1)^{(p-1)p/2}i_{\Lambda_0^\#(\varphi)}\frac{\Theta_0^n}{n!},
\end{equation*}
$d_{sp} : \Omega_{sb}^p(M) \to \Omega_{sb}^{p+1}(M)$ the operator which corresponds to each semi-basic form $\varphi$ the semi-basic part of its differential $d\varphi$, and $\delta = \ast \,d_{sb} \,\ast$ the associated ``codifferential'' operator on $\Omega_{sb}(M)=\oplus_{p\in \Z}\Omega_{sb}^p(M)$. By a straightforward, but long, computation, we show that (\ref{cond-sigma'}) is equivalent to the system
\begin{equation}\label{cond-sigma-tau}
\left\{
\begin{array}{l}
2\sigma\wedge \delta(\sigma)=\delta(\sigma\wedge \sigma)\\
\\
\delta(\sigma\wedge \tau) + \delta(\sigma)\wedge\tau - \sigma\wedge \delta(\tau) = (i_{\Lambda_0^\#(d\vartheta_0)}\sigma)\sigma -\frac{1}{2}i_{\Lambda_0^\#(d\vartheta_0)}(\sigma \wedge \sigma).
\end{array}
\right.
\end{equation}
Hence, we deduce:
\begin{proposition}\label{prop-sigma-odd}
Under the above assumptions and notations, $\Lambda = \Lambda_0^\#(\sigma) + \Lambda_0^\#(\tau)\wedge E_0$ defines a Poisson structure on $(M, \vartheta_0,\Theta_0)$ if and only if $(\sigma,\tau)$ satisfies (\ref{cond-sigma-tau}).
\end{proposition}

Concluding, we can announce the following theorem.
\begin{theorem}\label{THEOREM-ODD}
Let $f_1,\ldots,f_{2n+1-2k}$ be smooth functions on a $(2n+1)$-dimensional smooth manifold $M$ which are functionally independent almost everywhere, $(\vartheta_0,\Theta_0)$ an almost cosymplectic structure on $M$ such that (\ref{f-odd}) holds on an open and dense subset $\mathcal{U}$ of $M$, $\Omega = \vartheta_0\wedge\displaystyle{\frac{\Theta_0^n}{n!}}$ the corresponding volume form on $M$, and $(\sigma,\tau)$ an element of $\Omega^2_{sb}(M)\times \Omega^1_{sb}(M)$, with $(rank\,\sigma,\, rank\,\tau)=(2k,0)$ or $(2k,1)$ or $(2k-2,1)$ on $\mathcal{U}$, that has the properties (ii)-(iii) and satisfies (\ref{cond-sigma-tau}). Then, the bracket $\{\cdot,\cdot\}$ on $\C$ given, for any $h_1,h_2 \in \C$, by
\begin{equation}\label{br-odd}
\{h_1,h_2\}\Omega = - \frac{1}{f}dh_1\wedge dh_2\wedge (\sigma+ \frac{g}{k-1}\Theta_0)\wedge \frac{\Theta_0^{k-2}}{(k-2)!}\wedge df_1\wedge \ldots \wedge df_{2n+1-2k},
\end{equation}
where $f$ is that of (\ref{f-odd}) and $g = i_{\Lambda_0}\sigma$, defines a Poisson structure $\Lambda$ on $M$, $\Lambda = \Lambda_0^\#(\sigma) + \Lambda_0^\#(\tau)\wedge E_0$, with symplectic leaves of dimension at most $2k$ for which $f_1,\ldots,f_{2n+1-2k}$ are Casimirs. The converse is also true.
\end{theorem}

\begin{remark}
{\rm We remark that, in both cases (of even dimension $m=2n$ and of odd dimension $m=2n+1$), when $k=1$, the brackets (\ref{damianou:bracket-Lambda-Omega}) and (\ref{br-odd}) are reduced to a bracket of type (\ref{damianou:equation2}). Precisely,
\begin{equation*}
\{h_1,h_2\}\Omega = -\frac{g}{f}dh_1\wedge dh_2\wedge df_1\wedge \ldots \wedge df_{m-2}.
\end{equation*}}
\end{remark}

\section{Some Examples}
\subsection{Dirac Brackets}
Let $(M,\omega_0)$ be a symplectic manifold of dimension $2n$, $\Lambda_0$ its associated Poisson structure, and $f_1,\ldots, f_{2n-2k}$ smooth functions on $M$ whose the differentials are linearly independent at each point in the submanifold $M_0$ of $M$ defined by the equations $f_1(x) =0, \,\ldots, \, f_{2n-2k}(x) =0$. We assume that the matrix $\big(\{f_i,f_j\}_{_0} \big)$ is invertible on an open neighborhood $\mathcal{W}$ of $M_0$ in $M$ and we denote by $c_{ij}$ the coefficients of its inverse matrix which are smooth functions on $\mathcal{W}$ such that $\sum_{j=1}^{2n-2k}\{f_i,f_j\}_{_0}c_{jk} = \delta_{ik}$. We consider on $\mathcal{W}$ the $2$-form
\begin{equation}\label{damianou:sigma-dirac}
\sigma = \omega_0 + \sum_{i<j}c_{ij}df_i\wedge df_j.
\end{equation}
We will prove that it is a section of $\bigwedge^2 D^\circ$ of maximal rank on $\mathcal{W}$ which verifies (\ref{damianou:cond-delta-sigma}). As in subsection \ref{section-even}, $D$ denotes the subbundle of $TM$ generated by the hamiltonian vector fields $X_{f_i}$ of $f_i$, $i=1,\ldots,2n-2k$, with respect to $\Lambda_0$ and $D^\circ$ its annihilator. For any $X_{f_l}$, $l=1,\ldots,2n-2k$, we have
\begin{eqnarray*}
\sigma (X_{f_l},\cdot)& = & \omega_0(X_{f_l},\cdot) + \sum_{i<j}c_{ij}\langle df_i, X_{f_l}\rangle df_j - \sum_{i<j}c_{ij}\langle df_j, X_{f_l}\rangle df_i \\
& = & - df_l + \sum_{i<j}c_{ij}\{f_l,f_i\}_{_0}df_j   - \sum_{i<j}c_{ij}\{f_l,f_j\}_{_0}df_i \\
& = & -df_l + \sum_j\delta_{lj}df_j \,=\, -df_l + df_l\, =\, 0,
\end{eqnarray*}
which means that $\sigma$ is a section of $\bigwedge^2 D^\circ \to \mathcal{W}$. The assumption that $\big(\{f_i,f_j\}_{_0}\big)$ is invertible ensures  that $D$ is a symplectic subbundle of $T_{\mathcal{W}}M$. So, for any $x \in \mathcal{W}$, $T^\ast_x M = D_x^\circ \oplus \langle df_1,\ldots,df_{2n-2k} \rangle_x$ and $\bigwedge^2 T_x^\ast M = \bigwedge^2 D_x^\circ + \bigwedge^2 \langle df_1,\ldots,df_{2n-2k}\rangle_x + D_x^\circ \wedge \langle df_1,\ldots,df_{2n-2k}\rangle_x$. But, $\omega_0$ is a nondegenerate section of $\bigwedge^2 T^\ast M$ and the part $\sum_{i<j}c_{ij}df_i\wedge df_j$ of $\sigma$ is a smooth section of $\bigwedge^2 \langle df_1,\ldots,df_{2n-2k}\rangle$ of maximal rank on $\mathcal{W}$, because $\det(c_{ij})\neq 0$ on $\mathcal{W}$. Thus, $\sigma$ is of maximal rank on $\mathcal{W}$. Also, we have
\begin{equation*}
g = i_{\Lambda_0}\sigma = - \langle \omega_0 + \sum_{i<j}c_{ij}df_i\wedge df_j, \Lambda_0\rangle = - n - \sum_{i<j}c_{ij} \{f_i,f_j\}_{_0} = -n + (n-k) = -k,
\end{equation*}
and
\begin{eqnarray}\label{damianou:dirac-sigma}
\ast \, \sigma & \stackrel{(\ref{damianou:ast-Phi-sigma})(\ref{damianou:expression-Phi})}{=} & - \frac{1}{f}(\sigma + \frac{g}{k-1}\omega_0)\wedge \frac{\omega_0^{k-2}}{(k-2)!}\wedge df_1\wedge\ldots \wedge df_{2n-2k} \nonumber \\
& = & - \frac{1}{f} (\omega_0 + \sum_{i<j}c_{ij}df_i\wedge df_j - \frac{k}{k-1}\omega_0) \wedge \frac{\omega_0^{k-2}}{(k-2)!}\wedge df_1\wedge\ldots \wedge df_{2n-2k} \nonumber \\
& = & \frac{1}{f} \frac{\omega_0^{k-1}}{(k-1)!}\wedge df_1 \wedge \ldots \wedge df_{2n-2k}.
\end{eqnarray}
Consequently,
\begin{equation*}
\delta \sigma  =  (\ast \,d \,\ast)\sigma  \stackrel{(\ref{damianou:dirac-sigma})}{=} \ast (- \frac{df}{f}\wedge (\ast \,\sigma)) \stackrel{(\ref{damianou:property-ii})}{=} - \frac{1}{f} i_{X_f} \sigma
\end{equation*}
and
\begin{equation}\label{damianou:1}
\quad 2\sigma \wedge \delta(\sigma) = -\frac{2}{f}\sigma \wedge (i_{X_f}\sigma) = - \frac{1}{f}i_{X_f}(\sigma \wedge \sigma).
\end{equation}
On the other hand,
\begin{eqnarray}\label{damianou:dirac-sigma-sigma}
\ast \, (\sigma \wedge \sigma) & \stackrel{(\ref{damianou:property-ii})}{=} & - i_{\Lambda_0^\#(\sigma)} (\ast \, \sigma) \stackrel{(\ref{damianou:dirac-sigma})}{=} - \frac{1}{f} \big(i_{\Lambda_0^\#(\sigma)}\frac{\omega_0^{k-1}}{(k-1)!} \big)\wedge df_1 \wedge \ldots \wedge df_{2n-2k} \nonumber \\
& \stackrel{(\ref{damianou:eq-Lambda-k})}{=} & \frac{1}{f}\,[\ast (\sigma \wedge \frac{\omega_0^{n-k+1}}{(n-k+1)!})]\wedge df_1 \wedge \ldots \wedge df_{2n-2k} \nonumber \\
&\stackrel{(\ref{damianou:ast-sigma-k-1})(\ref{damianou:sigma-dirac})}{=} & - \frac{1}{f}(\omega_0 + \sum_{i<j}c_{ij}df_i\wedge df_j - \frac{k}{k-2}\omega_0)\wedge \frac{\omega_0^{k-3}}{(k-3)!}\wedge df_1 \wedge \ldots \wedge df_{2n-2k} \nonumber \\
& = & \frac{2}{f}\wedge \frac{\omega_0^{k-3}}{(k-3)!}\wedge df_1 \wedge \ldots \wedge df_{2n-2k}
\end{eqnarray}
and
\begin{equation}\label{damianou:2}
\delta (\sigma \wedge \sigma) = \ast \,d \,\ast (\sigma \wedge \sigma) \stackrel{(\ref{damianou:dirac-sigma-sigma})}{=} \ast \, \big(-\frac{df}{f}\wedge \ast\,(\sigma \wedge \sigma) \big) \stackrel{(\ref{damianou:property-ii})}{=} - \frac{1}{f}i_{X_f}(\sigma \wedge \sigma).
\end{equation}
From (\ref{damianou:1}) and (\ref{damianou:2}) we conclude that $\sigma$ verifies (\ref{damianou:cond-delta-sigma}). Thus, according to Theorem \ref{damianou:THEOREM}, the bivector field
\begin{equation*}
\Lambda = \Lambda_0^\#(\sigma) = \Lambda_0 + \sum_{i<j}c_{ij}X_{f_i}\wedge X_{f_j}
\end{equation*}
defines a Poisson structure on $\mathcal{W}$ whose  corresponding bracket $\{\cdot,\cdot\}$ on $C^\infty (\mathcal{W}, \mathbb{R})$ is given, for any $h_1,h_2 \in C^\infty (\mathcal{W},\mathbb{R})$, by
\begin{equation}\label{damianou:br-dirac}
\{h_1,h_2\}\Omega = \frac{1}{f} dh_1 \wedge dh_2 \wedge \frac{\omega_0^{k-1}}{(k-1)!}\wedge df_1 \wedge \ldots \wedge df_{2n-2k}.
\end{equation}
In the above expression of $\Lambda$ we recognize the Poisson structure defined by Dirac \cite{damianou:dr} on an open neighborhood $\mathcal{W}$ of the constrained submanifold $M_0$ of $M$ and in (\ref{damianou:br-dirac}) the expression of the Dirac bracket given in \cite{damianou:Grab2}.

\subsection{Almost Poisson brackets for nonholonomic systems}
Let $Q$ be the configuration space of a Lagrangian system with Lagrangian function $L:TQ\to \R$, subjected to nonholonomic, homogeneous, constraints defined by a distribution $C \subset TQ$ on $Q$. In a local coordinate system $(q^1,\ldots,q^n,\dot{q}^1,\ldots,\dot{q}^n)$ of $TQ$, $C$ is described by the independent equations
\begin{equation}\label{nh-C}
\zeta_s^i(q)\dot{q}^s =0,\footnote{In this subsection, the Einstein convention of sum over repeated indices holds.} \quad \quad i = 1,\ldots,n-k,
\end{equation}
where $\zeta^i_s$, $s=1,\ldots,n$, are smooth functions on $Q$, and the equations of motion of the nonholonomic system are given by
\begin{equation}\label{nh-eqmotion}
\frac{d}{dt}(\frac{\partial L}{\partial \dot{q}^s}) - \frac{\partial L}{\partial q^s} = \lambda_i\zeta^i_s, \quad s= 1,\ldots,n,
\end{equation}
($\lambda_i$ being the Lagrangian multipliers) together with the constraint equations (\ref{nh-C}).

We now turn  to the Hamiltonian formulation of our system on the cotangent bundle $T^\ast Q$ of $Q$. We suppose that $T^\ast Q$ is equipped with the standard, nondegenerate, Poisson structure $\Lambda_0 = \frac{\partial}{\partial p_s}\wedge \frac{\partial}{\partial q^s}$ associated with the symplectic form $\omega_0=dp_s\wedge dq^s$. Let $\mathcal{L}:TQ \to T^\ast Q$, $(q^s,\dot{q}^s)\mapsto (q^s,p_s=\frac{\partial L}{\partial \dot{q}^s})$, be the Legendre transformation associated with $L$. Assuming that $L$ is regular, we have that $\mathcal{L}$ is a diffeomorphism which maps the equations of motion (\ref{nh-eqmotion}) to the system
\begin{eqnarray}\label{nh-hamilton}
\dot{q}^s & = & \frac{\partial H}{\partial p_s} \nonumber \\
\dot{p}_s & = & -\frac{\partial H}{\partial q^s} + \lambda_i\zeta^i_s,\quad \quad s=1,\ldots,n,
\end{eqnarray}
where $H : T^\ast Q \to \R$ is the Hamiltonian given by $H=(\dot{q}^s\frac{\partial L}{\partial \dot{q}^s} - L)\circ \mathcal{L}^{-1}$, and the constraint distribution $C$ to the constraint submanifold $\mathcal{M}$ of $T^\ast Q$, which is defined by the equations
\begin{equation*}
f^i(q,p) = \zeta^i_s(q)\frac{\partial H}{\partial p_s} =0, \quad \quad i=1,\ldots,n-k.
\end{equation*}
Also, the regularity assumption on $L$ implies that, at each point $(q,p)\in \mathcal{M}$, $T_{(q,p)}T^\ast Q$ splits into a direct sum of symplectic subspace and that the matrix $\mathcal{C} = \big(\mathcal{C}^{ij}\big)=\big(\Lambda_0(df^i,\mathbf{q}^\ast \zeta^j)\big)=\big(\zeta^i_s\displaystyle{\frac{\partial^2H}{\partial p_s \partial p_t}}\zeta^j_t\big)$, which is symmetric, is invertible on $\mathcal{M}$. Precisely,
\begin{equation*}
T_{(q,p)}T^\ast Q = T_{(q,p)}\mathcal{M}\oplus \mathcal{Z},
\end{equation*}
where $\mathcal{Z}\subset TT^\ast Q$ is the distribution on $T^\ast Q$ spanned by the vector fields
\begin{equation*}
Z^i = \zeta_s^i\frac{\partial}{\partial p_s} = \Lambda_0^\#(-\mathbf{q}^\ast \zeta^i),
\end{equation*}
where $\zeta^i = \zeta^i_s(q)dq^s$, $i=1,\ldots,n-k$, are the constraint $1$-forms on $Q$ and $\mathbf{q} : T^\ast Q \to Q$ is the canonical projection. Hence, in view of (\ref{nh-hamilton}), the Hamiltonian vector field $X_H = \Lambda_0^\#(dH)$ admits, along $\mathcal{M}$, the decomposition $X_H = X_{nh}-\lambda_iZ^i$. The part $X_{nh}$ is tangent to $\mathcal{M}$ and $\lambda_iZ^i$ lies on $\mathcal{Z}$, along $\mathcal{M}$. According to the results of \cite{c-mdl-d, mrl-nh, vm}, the dynamical equations of $X_{nh}$ on $\mathcal{M}$ are expressed in Hamiltonian form with respect to the restriction $\{\cdot,\cdot\}_{nh}^{\mathcal{M}}$ on $C^\infty(\mathcal{M})$ of the bracket $\{\cdot,\cdot\}_{nh}$ given, for any $H_1,H_2\in C^\infty(T^\ast Q)$, by
\begin{eqnarray}\label{nh-bracket}
\lefteqn{\{H_1,H_2\}_{nh} = \{H_1,H_2\}_{_0} +\mathcal{C}_{lm}\{f^l,H_1\}_{_0}\langle dH_2,Z^m\rangle} \nonumber \\
& & - \mathcal{C}_{lm}\{f^l,H_2\}_{_0}\langle dH_1,Z^m\rangle +  \mathcal{C}_{ij}\{f^j,f^l\}_{_0}\mathcal{C}_{lm}\langle dH_1,Z^i\rangle \langle dH_2,Z^m\rangle,
\end{eqnarray}
where $\{\cdot,\cdot\}_{_0}$ is the bracket of $\Lambda_0$ on $C^\infty(T^\ast Q)$ and $\big(\mathcal{C}_{ij}\big)$ is the inverse matrix of $\mathcal{C}$. In other words, for functions $h_1,h_2 \in C^\infty(\mathcal{M})$, the value of $\{h_1,h_2\}_{nh}^{\mathcal{M}}$ is equal to the value of $\{H_1,H_2\}_{nh}$ along $\mathcal{M}$, where $H_1$ and $H_2$ are, respectively, arbitrary smooth extensions of $h_1$ and $h_2$ on $T^\ast Q$. We will show that (\ref{nh-bracket}) holds, and so $\{\cdot,\cdot\}_{nh}^{\mathcal{M}}$, can be calculated by (\ref{br-almostPoisson}).

We remark that
\begin{equation*}
\Lambda_{nh} = \Lambda_0 + \mathcal{C}_{lm}X_{f^l}\wedge Z^m + \frac{1}{2}\mathcal{C}_{ij}\{f^j,f^l\}_{_0} \ \mathcal{C}_{lm}Z^i\wedge Z^m,
\end{equation*}
where $X_{f^l}=\Lambda_0^\#(df^l)$, is the bivector field on $T^\ast Q$ associated to (\ref{nh-bracket}) whose the kernel along $\mathcal{M}$ coincides with the space $\langle df^1,\ldots, df^{n-k}, \mathbf{q}^\ast \zeta^1,\ldots,\mathbf{q}^\ast \zeta^{n-k}\rangle \vert_{\mathcal{M}}$. In fact,
\begin{eqnarray*}
\Lambda_{nh}(df^s) & = & X_{f^s} + \mathcal{C}_{lm}\{f^l,f^s\}_{_0}Z^m - \mathcal{C}_{lm}\langle df^s,Z^m\rangle X_{f^l} \nonumber \\
& & + \, \frac{1}{2}\mathcal{C}_{ij}\{f^j,f^l\}_{_0}\mathcal{C}_{lm} \langle df^s, Z^i\rangle Z^m - \frac{1}{2}\mathcal{C}_{ij}\{f^j,f^l\}_{_0}\mathcal{C}_{lm}\langle df^s,Z^m\rangle Z^i \nonumber \\
& = & X_{f^s} + \mathcal{C}_{lm}\{f^l,f^s\}_{_0}Z^m - \mathcal{C}_{lm}\mathcal{C}^{sm} X_{f^l} \nonumber \\
& & +\, \frac{1}{2}\mathcal{C}_{ij}\{f^j,f^l\}_{_0}\mathcal{C}_{lm}\mathcal{C}^{si}Z^m - \frac{1}{2}\mathcal{C}_{ij}\{f^j,f^l\}_{_0}\mathcal{C}_{lm}\mathcal{C}^{sm}Z^i \nonumber \\
& = & X_{f^s} + \mathcal{C}_{lm}\{f^l,f^s\}_{_0}Z^m - X_{f^s} + \frac{1}{2}\{f^s,f^l\}_{_0}\mathcal{C}_{lm}Z^m - \frac{1}{2}\mathcal{C}_{ij}\{f^j,f^s\}_{_0}Z^i\\
& = & 0
\end{eqnarray*}
and
\begin{equation*}
\Lambda_{nh}(\mathbf{q}^\ast \zeta^s)  = \Lambda_0^\#(\mathbf{q}^\ast \zeta^s) +  \mathcal{C}_{lm}\langle \mathbf{q}^\ast \zeta^s, X_{f^l}\rangle Z^m = - Z^s + \mathcal{C}_{lm}\mathcal{C}^{ls}Z^m = - Z^s + Z^s = 0,
\end{equation*}
while $\mathrm{rank}\, \Lambda_{nh} = 2k$ everywhere on $\mathcal{M}$ \cite{vm}. On the other hand, $\Lambda_{nh}$ can be viewed as the image, via the isomorphism $\Lambda_0^\#$, of the $2$-form
\begin{equation*}
\sigma = \omega_0 - \mathcal{C}_{lm}df^l \wedge \mathbf{q}^\ast \zeta^m + \frac{1}{2}\mathcal{C}_{ij}\{f^j,f^l\}_{_0}\mathcal{C}_{lm}\mathbf{q}^\ast \zeta^i\wedge \mathbf{q}^\ast \zeta^m
\end{equation*}
on $T^\ast Q$ with $\mathrm{rank}\,\sigma = 2k$ on $\mathcal{M}$. Also,
\begin{equation*}
f= \langle df^1\wedge\ldots \wedge df^{n-k}\wedge \mathbf{q}^\ast \zeta^1 \wedge \ldots \wedge \mathbf{q}^\ast \zeta^{n-k},\; \frac{\Lambda_0^{n-k}}{(n-k)!} \rangle \neq 0
\end{equation*}
on $\mathcal{M}$, because $f^2 = \det J = \det \mathcal{C}^2 \neq 0$ on $\mathcal{M}$, where
\begin{equation*}
J= \left(\begin{array}{cc}
\{f^i,f^j\}_{_0} & \Lambda_0(df^i,\mathbf{q}^\ast \zeta^j) \\
\Lambda_0(\mathbf{q}^\ast \zeta^i, df^j) & \Lambda_0(\mathbf{q}^\ast \zeta^i,\mathbf{q}^\ast \zeta^j)
\end{array}
\right) =
\left( \begin{array}{cc}
\{f^i,f^j\}_{_0} & \mathcal{C} \\
-\mathcal{C} & 0
\end{array}
\right),
\end{equation*}
and
\begin{eqnarray*}
g & = & i_{\Lambda_0}\sigma = - \langle \omega_0 - \mathcal{C}_{lm}df^l \wedge \mathbf{q}^\ast \zeta^m + \frac{1}{2}\mathcal{C}_{ij}\{f^j,f^l\}_{_0}\mathcal{C}_{lm}\mathbf{q}^\ast \zeta^i\wedge \mathbf{q}^\ast \zeta^m, \, \Lambda_0 \rangle \\
& = &  - (n- \mathcal{C}_{lm}\mathcal{C}^{lm}) = - n + (n-k) = -k.
\end{eqnarray*}
Hence, we can apply (\ref{br-almostPoisson}) for the calculation of $\{\cdot,\cdot\}_{nh}$ on $C^\infty(T^\ast Q)$ and, by restriction, on $C^\infty(\mathcal{M})$. For any $H_1,H_2 \in C^\infty(T^\ast Q)$,
\begin{equation*}
\{H_1,H_2\}_{nh} \Omega = \frac{1}{f}dH_1 \wedge dH_2 \wedge \frac{\omega_0^{k-1}}{(k-1)!}\wedge df^1\wedge \ldots \wedge df^{n-k}\wedge \mathbf{q}^\ast \zeta^1 \wedge \ldots \wedge \mathbf{q}^\ast \zeta^{n-k},
\end{equation*}
where $\Omega = \displaystyle{\frac{\omega_0^n}{n!}}$ is the corresponding volume element on $T^\ast Q$.
\begin{remark}
{\rm Without doubt, $\Lambda_{nh}$ is Poisson if and only if $\sigma$ satisfies (\ref{damianou:cond-delta-sigma}). But, Van der Schaft and Maschke proved \cite{vm} that $\{\cdot,\cdot\}_{nh}$ satisfies the Jacobi identity if and only if the constraints (\ref{nh-C}) are holonomic. Hence, we conclude that $\sigma$ satisfies (\ref{damianou:cond-delta-sigma}) if and only if the constraint distribution $C$ is completely integrable. These facts have an interesting geometric interpretation observed by Koon and Marsden \cite{km}; the vanishing of the Schouten bracket $[\Lambda_{nh},\Lambda_{nh}]$ is equivalent with the vanishing of the curvature of an Ehresmann connection associated with the constraint distribution $C$.}
\end{remark}

\subsection{Periodic Toda and Volterra lattices}
In this paragraph we  study the linear Poisson structure $\Lambda_{_T}$ associated with  the periodic Toda lattice of $n$ particles. This Poisson structure has two well-known Casimir functions. Using  Theorem \ref{damianou:THEOREM} we construct another Poisson structure having the same Casimir invariants with $\Lambda_{_T}$. It turns out that this structure decomposes as a direct sum of two Poisson tensors one of which (involving only the $a$ variables in Flaschka's coordinates) is the quadratic Poisson bracket of the Volterra lattice (also known as the KM-system). It agrees with the general philosophy (see \cite{damianou:Dam02}) that one obtains the Volterra lattice from the Toda lattice by restricting to the $a$ variables.

The periodic Toda lattice of $n$ particles ($n\geq 2$) is the system of ordinary differential equations on $\mathbb{R}^{2n}$ which in Flaschka's \cite{damianou:fl} coordinate system $(a_1,\ldots,a_n, b_1,\ldots,b_n)$ takes the form
\begin{equation*}
\dot{a}_i = a_i(b_{i+1} - b_i) \quad \mathrm{and} \quad \dot{b}_i = 2(a_i^2 - a_{i-1}^2) \quad \quad (i \in \mathbb{Z} \;\;\; \mathrm{and} \;\;\; (a_{i+n},b_{i+n})=(a_i,b_i)).
\end{equation*}
This system is hamiltonian with respect to the nonstandard Lie-Poisson structure
\begin{equation*}\label{damianou:Poisson-Toda}
\Lambda_T = \sum_{i=1}^n a_i \frac{\partial}{\partial a_i}\wedge(\frac{\partial}{\partial b_i} - \frac{\partial}{\partial b_{i+1}})
\end{equation*}
on $\mathbb{R}^{2n}$ and it has as hamiltonian the function $H = \sum_{i = 1}^n(a_i^2 + \displaystyle{\frac{1}{2}b_i^2})$. The structure $\Lambda_T$ is of rank $2n-2$ on $\mathcal{U} = \{(a_1,\ldots,a_n,b_1,\ldots,b_n) \in \mathbb{R}^{2n} \; / \; \sum_{i=1}^n a_1\ldots a_{i-1}a_{i+1}\ldots a_n \neq 0\}$ and it admits two Casimir functions:
\begin{equation*}\label{damianou:Toda-Casimir}
C_1 = b_1 + b_2 + \ldots + b_n   \quad \quad \mathrm{and} \quad \quad C_2 = a_1a_2\ldots a_n.
\end{equation*}

We consider on $\mathbb{R}^{2n}$ the standard symplectic form $\omega_0 = \sum_{i=1}^n da_i \wedge db_i$, its associated Poisson tensor $\Lambda_0 = \sum_{i=1}^n\displaystyle{\frac{\partial}{\partial a_i}\wedge \frac{\partial}{\partial b_i}}$, and the corresponding volume element $\Omega = \displaystyle{\frac{\omega_0^n}{n!}}=da_1\wedge db_1\wedge \ldots \wedge da_n \wedge db_n$. The hamiltonian vector fields of $C_1$ and $C_2$ with respect to $\Lambda_0$ are
\begin{equation*}
X_{_{C_1}} = - \sum_{i=1}^n \frac{\partial}{\partial a_i} \quad \quad \mathrm{and} \quad \quad X_{_{C_2}} =\sum_{i=1}^n a_1\ldots a_{i-1}a_{i+1}\ldots a_n \frac{\partial}{\partial b_i}.
\end{equation*}
So, $D = \langle X_{_{C_1}}, X_{_{C_2}}\rangle$ and
\begin{equation*}
D^\circ = \big\{\sum_{i=1}^n(\alpha_i da_i + \beta_i db_i)\in \Omega^1(\mathbb{R}^{2n})\; / \; \sum_{i=1}^n\alpha_i = 0 \;\;\; \mathrm{and} \;\;\; \sum_{i=1}^n a_1\ldots a_{i-1}\beta_ia_{i+1}\ldots a_n = 0\big\}.
\end{equation*}
The family of $1$-forms $(\sigma_1,\ldots,\sigma_{n-1},\sigma'_1,\ldots,\sigma'_{n-1})$,
\begin{equation*}
\sigma_j = da_j - da_{j+1} \quad \quad \mathrm{and} \quad \quad \sigma'_{j} = a_jdb_j - a_{j+1}db_{j+1}, \quad \quad j = 1,\ldots, n-1,
\end{equation*}
provides, at every point $(a,b) \in \mathcal{U}$, a basis of $D^\circ_{(a,b)}$. The section of maximal rank $\sigma_{_T}$ of $\bigwedge^2D^\circ \to \mathcal{U}$, which corresponds to $\Lambda_T$, via the isomorphism $\Lambda_0^\#$, and verifies (\ref{damianou:cond-delta-sigma}), is written, in this basis, as
\begin{equation*}
\sigma_{_T} = \sum_{j = 1}^{n-1} \sigma_j \wedge \big(\sum_{l=j}^{n-1}\sigma'_l\big).
\end{equation*}

Now, we consider on $\mathbb{R}^{2n}$ the $2$-form
\begin{eqnarray*}
\sigma & = & \sum_{j=1}^{n-2} \sigma_j \wedge \big(\sum_{l=j+1}^{n-1}\sigma_l\big) \,+ \,\sum_{j=1}^{n-2} \sigma'_j \wedge \big(\sum_{l=j+1}^{n-1}\sigma'_l\big) \nonumber \\
& = & \sum_{j=1}^{n-2} \big[(da_j - da_{j+1}) \wedge (da_{j+1} - da_n)  +  (a_jdb_j - a_{j+1}db_{j+1}) \wedge (a_{j+1}db_{j+1} - a_ndb_n)\big] \nonumber \\
& = & \sum_{j=1}^{n}\big(da_j \wedge da_{j+1}  + a_ja_{j+1}db_j \wedge db_{j+1}\big).
\end{eqnarray*}
It is a section of $\bigwedge^2D^\circ$ whose  rank depends on  the parity of $n$; if $n$ is odd, its rank is $2n-2$ on $\mathcal{U}$, while, if $n$ is even, its rank is $2n-4$ almost everywhere on $\mathbb{R}^{2n}$. Also, after a long computation, we can confirm that it satisfies (\ref{damianou:cond-delta-sigma}). Thus, its image via $\Lambda_0^\#$, i.e., the bivector field
\begin{equation}\label{new-poisson-vol}
\Lambda = \sum_{j=1}^{n}\big(a_ja_{j+1}\frac{\partial}{\partial a_j}\wedge \frac{\partial}{\partial a_{j+1}} + \frac{\partial}{\partial b_j}\wedge \frac{\partial}{\partial b_{j+1}}\big),
\end{equation}
defines a Poisson structure on $\mathbb{R}^{2n}$ with symplectic leaves of dimension at most $2n-2$, when $n$ is odd, that has $C_1$ and $C_2$ as Casimir functions. (When $n$ is even, $\Lambda$ has two more Casimir functions.) We remark that $(\mathbb{R}^{2n},\Lambda)$ can be viewed as the product of Poisson manifolds $(\mathbb{R}^n, \Lambda_{_V})\times (\mathbb{R}^n, \Lambda')$, where
\begin{equation*}
\Lambda_{_V} = \sum_{j=1}^{n}a_ja_{j+1}\frac{\partial}{\partial a_j}\wedge \frac{\partial}{\partial a_{j+1}}  \quad \quad \mathrm{and} \quad \quad \Lambda'= \sum_{j=1}^{n}\frac{\partial}{\partial b_j}\wedge \frac{\partial}{\partial b_{j+1}}.
\end{equation*}
The Poisson tensor $\Lambda_{_V}$ is the quadratic bracket of the  periodic Volterra lattice on $\mathbb{R}^n$ and it has $C_2$ as unique Casimir function, when $n$ is odd.

\vspace{3mm}

In the following, using (\ref{damianou:bracket-Lambda-Omega}), we illustrate the explicit formul{\ae} of the brackets of $\Lambda_{_T}$ and $\Lambda$ in the special case $n=3$. We have $C_1 = b_1 + b_2 + b_3$, $C_2 = a_1a_2a_3$, $k = 2$, $\Lambda_0 = \sum_{i=1}^3\displaystyle{\frac{\partial}{\partial a_i}\wedge \frac{\partial}{\partial b_i}}$, and $\Omega = da_1\wedge db_1 \wedge da_2 \wedge db_2 \wedge da_3 \wedge db_3$. Consequently, $f = \langle dC_1\wedge dC_2, \,\Lambda_0\rangle = -(a_1a_2 + a_2a_3 +a_1a_3)$, which is a nonvanishing function on $\mathcal{U}$.

For the periodic Toda lattice of $3$ particles, we have $\sigma_{_T} = (da_1 - da_2)\wedge (a_1db_1 -a_3db_3) + (da_2 - da_3)\wedge (a_2db_2 - a_3db_3)$, $g_{_T} = i_{\Lambda_0}\sigma_{_T} = -(a_1 + a_2 +a_3)$ and
\begin{eqnarray*}
\Phi_{_T} & = & - \frac{1}{f}(\sigma_{_T} + g_{_T}\omega_0)\wedge dC_1 \wedge dC_2 \\
          & = & -a_1 db_1 \wedge da_2 \wedge da_3 \wedge db_3 + a_1 da_2 \wedge db_2 \wedge da_3 \wedge db_3 + a_2 da_1 \wedge db_1 \wedge da_3 \wedge db_3 \\
          & & -a_2 da_1 \wedge db_1 \wedge db_2 \wedge da_3 + a_3 da_1 \wedge da_2 \wedge db_2 \wedge db_3 + a_3da_1 \wedge db_1 \wedge da_2 \wedge db_2.
\end{eqnarray*}
Thus,
\begin{equation*}
\begin{array}{ll}
\{a_1,b_1\}_{_T}\Omega = da_1 \wedge db_1 \wedge \Phi_{_T} = a_1 \Omega, & \{a_1,b_2\}_{_T}\Omega = da_1 \wedge db_2 \wedge \Phi_{_T} = - a_1 \Omega, \\
\\
\{a_2,b_2\}_{_T}\Omega = da_2 \wedge db_2 \wedge \Phi_{_T} = a_2 \Omega, & \{a_2,b_3\}_{_T}\Omega = da_2 \wedge db_3 \wedge \Phi_{_T} = - a_2 \Omega,\\
\\
\{a_3,b_3\}_{_T}\Omega = da_3 \wedge db_3 \wedge \Phi_{_T} = a_3 \Omega, & \{a_3,b_1\}_{_T}\Omega = da_3 \wedge db_1 \wedge \Phi_{_T} = - a_3 \Omega,
\end{array}
\end{equation*}
and all other brackets are zero.

For the Poisson structure (\ref{new-poisson-vol}) on $\R^6$, we have $\sigma = (da_1 - da_2)\wedge (da_2 - da_3) + (a_1db_1 -a_2db_2)\wedge (a_2db_2 -a_3db_3)$, $g = i_{\Lambda_0}\sigma = 0$ and
\begin{eqnarray*}
\Phi & = & - \frac{1}{f}\sigma \wedge dC_1 \wedge dC_2 \\
          & = & -a_1a_2db_1 \wedge db_2 \wedge da_3 \wedge db_3 + a_1a_3 db_1 \wedge da_2 \wedge db_2 \wedge db_3 - a_2a_3 da_1 \wedge db_1 \wedge db_2 \wedge db_3 \\
          & & - da_1 \wedge db_1 \wedge da_2 \wedge da_3 - da_1 \wedge da_2 \wedge da_3 \wedge db_3 + da_1 \wedge da_2 \wedge db_2 \wedge da_3.
\end{eqnarray*}
Thus,
\begin{equation*}
\begin{array}{ll}
\{a_1,a_2\}\Omega = da_1 \wedge da_2 \wedge \Phi = a_1a_2 \Omega, & \{a_1,a_3\}\Omega = da_1 \wedge da_3 \wedge \Phi = - a_1a_3 \Omega,\\
\\
\{a_2,a_3\}\Omega = da_2 \wedge da_3 \wedge \Phi = a_2a_3 \Omega, & \{b_1,b_2\}\Omega = db_1 \wedge db_2 \wedge \Phi = \Omega,\\
\\
\{b_1,b_3\}\Omega = db_1 \wedge db_3 \wedge \Phi = - \Omega, & \{b_2,b_3\}\Omega = db_2 \wedge db_3 \wedge \Phi = \Omega,
\end{array}
\end{equation*}
and all other brackets are zero.

\subsection{A Lie-Poisson bracket on $\mathbf{\mathfrak{gl}(3,\R)}$}
On the $9$-dimensional space $\mathfrak{gl}(3,\R)$ of $3\times 3$ matrices
\begin{equation*}
\left(
\begin{array}{ccc}
x_1 & z_2 & y_3 \\
y_1 & x_2 & z_3 \\
z_1 & y_2 & x_3
\end{array}
\right),
\end{equation*}
which is isomorphic to $\R^9$, we consider the functions
\begin{equation*}
C_1(x,y,z) = x_1 + x_2 + x_3, \quad C_2(x,y,z) = y_1z_2 + y_2z_3 + y_3z_1 \quad \mathrm{and} \quad C_3(x,y,z) = z_1z_2z_3.
\end{equation*}
 Using  Theorem \ref{THEOREM-ODD}, we are able  to construct a linear Poisson structure $\Lambda$ on $\mathfrak{gl}(3,\R)$, with sysmplectic leaves of dimension at most $6$, having $C_1$, $C_2$ and $C_3$ as Casimir functions. For this, we consider on $\mathfrak{gl}(3,\R)\cong \R^9$ the cosymplectic structure $(\vartheta_0,\Theta_0)$,
\begin{equation*}
\vartheta_0 = dz_3 \quad \mathrm{and} \quad \Theta_0 = dx_1 \wedge dy_1 + dx_2\wedge dy_2 + dx_3 \wedge dy_3 + dz_1 \wedge dz_2,
\end{equation*}
whose  corresponding transitive Jacobi structure $(\Lambda_0,E_0)$ is:
\begin{equation*}
\Lambda_0 = \frac{\partial}{\partial x_1}\wedge \frac{\partial}{\partial y_1} + \frac{\partial}{\partial x_2}\wedge \frac{\partial}{\partial y_2} + \frac{\partial}{\partial x_3}\wedge \frac{\partial}{\partial y_3} + \frac{\partial}{\partial z_1}\wedge \frac{\partial}{\partial z_2} \quad \mathrm{and} \quad E_0 = \frac{\partial}{\partial z_3}.
\end{equation*}
Clearly,
\begin{equation*}
f = \langle dC_1\wedge dC_2\wedge dC_3,\, E_0\wedge\Lambda_0\rangle = -z_1z_2^2 - z_1^2z_2 - z_1z_2z_3
\end{equation*}
is nonzero on the open and dense subset $\mathcal{U} = \{(x,y,z)\in \R^9 \,/\, z_1z_2^2 + z_1^2z_2 + z_1z_2z_3 \neq 0\}$ of $\mathfrak{gl}(3,\R)\cong \R^9$ and
\begin{equation*}
\Omega = \vartheta_0 \wedge \Theta_0^4 = dx_1\wedge dy_1\wedge dx_2\wedge dy_2 \wedge dx_3\wedge dy_3 \wedge dz_1\wedge dz_2 \wedge dz_3
\end{equation*}
is a volume form of $\mathfrak{gl}(3,\R)$. Furthermore, we consider on $\mathfrak{gl}(3,\R)$ the pair of semi-basic forms $(\sigma, \tau)$,
\begin{eqnarray*}
\sigma & = & -z_1dx_1\wedge dx_2 -z_2 dx_2\wedge dx_3 + z_3 dx_1 \wedge dx_3 -y_1dx_1\wedge dy_1 + y_1dx_1\wedge dy_2 \\
& & - y_2dx_2 \wedge dy_2 + y_2dx_2\wedge dy_3 -y_3dx_3\wedge dy_3 +  y_3dx_3\wedge dy_1 \\
& & - z_2dy_1\wedge dz_1 -z_1dy_1\wedge\wedge dz_2 + z_2 dy_2\wedge dz_1 + z_1 dy_3\wedge dz_2
\end{eqnarray*}
and
\begin{equation*}
\tau = -z_3dy_2 + z_3 dy_3,
\end{equation*}
which has the properties (ii)-(iii) and verifies the system (\ref{cond-sigma-tau}). Thus, the bracket $\{\cdot,\cdot \}$ on $C^\infty(\mathfrak{gl}(3,\R))$ given by (\ref{br-odd}) defines a Poisson structure $\Lambda$ on $\mathfrak{gl}(3,\R)$. We have, $g = i_{\Lambda_0}\sigma = y_1 + y_2 + y_3$ and
\begin{eqnarray*}
\Phi & = & -\frac{1}{f}(\sigma + \frac{g}{2}\Theta_0)\wedge \Theta_0 \wedge dC_1 \wedge dC_2 \wedge dC_3 \\
     & = & z_1dx_1\wedge dy_1\wedge dx_2\wedge dy_2 \wedge dy_3 \wedge dz_2 \wedge dz_3  -z_1 dy_1\wedge dx_2\wedge dy_2 \wedge dx_3\wedge dy_3 \wedge dz_2 \wedge dz_3 \\
     &- & z_1 dx_1\wedge dx_2\wedge dx_3\wedge dy_3 \wedge dz_1\wedge dz_2 \wedge dz_3 - z_2 dx_1 \wedge dy_1 \wedge dx_2 \wedge dx_3 \wedge dz_1 \wedge dz_2 \wedge dz_3   \\
     & -&  z_2 dy_1\wedge  dx_2\wedge dy_2\wedge dx_3\wedge dy_3\wedge dz_1\wedge dz_3 +z_2 dy_1\wedge dz_1\wedge dx_3\wedge dy_3\wedge dz_3\wedge dy_2\wedge dx_1 \\
     & -& y_1 dx_3\wedge dy_3\wedge dz_1\wedge dz_2\wedge dz_3\wedge dy_2\wedge dx_2 -y_3 dy_1\wedge dz_1\wedge dx_2\wedge dy_2\wedge dz_3\wedge dz_2\wedge dx_3 \\
     &- & y_1 dx_1\wedge dy_2\wedge dz_1\wedge dz_2\wedge dz_3\wedge dy_3\wedge dx_3 -z_3 dy_2\wedge dz_1\wedge dx_1\wedge dy_1\wedge dz_2\wedge dy_3\wedge dx_2 \\
     &- & y_2 dx_2\wedge dy_3\wedge dz_1\wedge dz_2\wedge dz_3\wedge dy_1\wedge dx_1+z_3 dx_1\wedge dx_2\wedge dz_1\wedge dz_2\wedge dz_3\wedge dy_2\wedge dx_3\\
     &- & y_3 dy_1\wedge dz_1\wedge dx_2\wedge dy_2\wedge dz_3\wedge dz_2\wedge dx_1-z_3 dy_2\wedge dz_1\wedge dx_3\wedge dy_3\wedge dz_2\wedge dy_1\wedge dx_1\\
     &- & y_2 dy_3\wedge dz_2\wedge dx_1\wedge dy_1\wedge dz_1\wedge dz_3\wedge dx_3.
\end{eqnarray*}
So,
\begin{equation*}
\begin{array}{ll}
\{x_1,y_1\}\Omega = dx_1\wedge dy_1\wedge \Phi = -y_1\Omega, & \{x_1,y_3\}\Omega = dx_1\wedge dy_3\wedge \Phi = y_3\Omega,\\
\\
\{x_1,z_1\}\Omega = dx_1\wedge dz_1\wedge \Phi = -z_1\Omega, &  \{x_1,z_2\}\Omega = dx_1\wedge dz_2\wedge \Phi = z_2\Omega,\\
\\
\{x_2,y_1\}\Omega = dx_2\wedge dy_1\wedge \Phi = y_1\Omega, &  \{x_2,y_2\}\Omega = dx_2\wedge dy_2\wedge \Phi = -y_2\Omega,\\
\\
\{x_2,z_2\}\Omega = dx_2\wedge dz_2\wedge \Phi = -z_2\Omega, &  \{x_2,z_3\}\Omega = dx_2\wedge dz_3\wedge \Phi = z_3\Omega,\\
\\
\{x_3,y_2\}\Omega = dx_3\wedge dy_2\wedge \Phi = y_2\Omega, & \{x_3,y_3\}\Omega = dx_3\wedge dy_3\wedge \Phi = -y_3\Omega,\\
\\
\{x_3,z_1\}\Omega = dx_3\wedge dz_1\wedge \Phi = z_1\Omega, & \{x_3,z_3\}\Omega = dx_3\wedge dz_3\wedge \Phi = -z_3\Omega,\\
\\
\{y_1,y_2\}\Omega = dy_1\wedge dy_2\wedge \Phi = -z_1\Omega, &  \{y_1,y_3\}\Omega = dy_1\wedge dy_3\wedge \Phi = z_3\Omega,\\
\\
\{y_2,y_3\}\Omega = dy_2\wedge dy_3\wedge \Phi = -z_2\Omega,&
\end{array}
\end{equation*}
and all other brackets are zero.

The Lie-Poisson bracket in this example coincides with the one of the bi--Hamiltonian pair formulated by Meucci \cite{meu} for $\mathrm{Toda}_3$ system, a dynamical system studied by Kupershmidt in \cite{kup} as a reduction of the KP hierarchy. Meucci derives this structure by a suitable restriction of a related pair of Lie algebroids on the set of maps from the cyclic group $\mathbb{Z}_3$ to $\mathrm{GL}(3, \R)$. Explicit formul{\ae} for the above bracket can also be found in \cite{dam-magri} where the $\mathrm{Toda}_3$ system is reduced to the phase space of the full Kostant--Toda lattice.

\subsection*{Acknowledgements}
The authors would like to thank Professor G.  Marmo for pointing out references \cite{damianou:Grab93, damianou:Grab2} where the basic ideas of this work lie.

\vspace{5mm}
\noindent
Pantelis A. DAMIANOU
\\
\emph{Department of Mathematics and Statistics}, \emph{University of Cyprus} \\ \emph{P.O. Box 20537, 1678 Nicosia, Cyprus} \\
\noindent\emph{E-mail: damianou@ucy.ac.cy}
\\
\\
\noindent
Fani PETALIDOU
\\
\emph{Department of Mathematics and Statistics}, \emph{University of Cyprus} \\
\emph{P.O. Box 20537, 1678 Nicosia, Cyprus} \\
\noindent\emph{E-mail: petalido@ucy.ac.cy}
\end{document}